\documentclass[12pt]{amsart}

\usepackage[dvipsnames,svgnames,x11names]{xcolor}
\usepackage{amsfonts,latexsym,amsthm,mathrsfs,amssymb,amsmath,amscd,euscript,tikz}
\usepackage{empheq}
\usepackage{color}
\usepackage[latin1]{inputenc}
\usepackage{comment}

\usepackage{calrsfs}
\DeclareMathAlphabet{\pazocal}{OMS}{zplm}{m}{n}

\usepackage{hyperref}
\usepackage[alphabetic, msc-links, backrefs, bibtex-style]{amsrefs}
\usepackage{stackengine}
\usepackage{framed}
\usepackage{xfrac}
\usepackage{esvect}
\usepackage[makeroom]{cancel}
\usepackage{faktor}
\usepackage{pgf,tikz,pgfplots,siunitx}
\usetikzlibrary{arrows,backgrounds, intersections, positioning, patterns, quotes}
\pgfplotsset{compat=1.17}
\usepgfplotslibrary{fillbetween}
\usepackage{mathrsfs}
\usepackage{enumerate}
\usepackage{dirtytalk}
\usepackage{alphabeta}
\usepackage{setspace}
\definecolor{wrwrwr}{rgb}{0.3803921568627451,0.3803921568627451,0.3803921568627451}
\definecolor{rvwvcq}{rgb}{0.08235294117647059,0.396078431372549,0.7529411764705882}
\usepackage{hyperref}
    \hypersetup{colorlinks=true,citecolor=blue,urlcolor =black,linkbordercolor={1 0 0}}
\usepackage{float}

\allowdisplaybreaks[1]

\usepackage{lmodern}
\usepackage{url,verbatim,enumerate,stmaryrd,bbm}
\usepackage{alltt}
\usepackage{amsthm}
\usepackage{booktabs}
\usepackage{caption}
\usepackage{fancyhdr}
\usepackage{graphicx}
\usepackage{mathdots}
\usepackage{mathtools}
\usepackage{multirow}
\usepackage{pdflscape}
\usepackage{pgfplots}
\usepackage{siunitx}
\usepackage{textcomp}
\usepackage{slashed}
\usepackage{tabularx}
\usepackage{tikz}
\usepackage{tkz-euclide}
\usepackage[normalem]{ulem}
\usepackage[all]{xy}
\usepackage{imakeidx}
\newcommand{\R}{\mathbb{R}}
\definecolor{verdeosc}{rgb}{0,0.6,0}

\newcommand{\ee}{\mathbf{e}}
\newcommand{\bv}{\mathbf{v}}
\usepackage{tikz,tikz-cd}
\usetikzlibrary{matrix,calc,positioning,arrows,decorations.pathreplacing,patterns,knots}
\usetikzlibrary{arrows.meta}
\usetikzlibrary{decorations.markings}
\usetikzlibrary{decorations.pathmorphing}
\usetikzlibrary{positioning}
\usetikzlibrary{fadings}
\usetikzlibrary{intersections}
\usetikzlibrary{cd}

\pgfarrowsdeclarecombine{twolatex'}{twolatex'}{latex'}{latex'}{latex'}{latex'}
\tikzset{->/.style = {decoration={markings,
                                  mark=at position 1 with {\arrow[scale=2]{latex'}}},
                      postaction={decorate}}}
\tikzset{<-/.style = {decoration={markings,
                                  mark=at position 0 with {\arrowreversed[scale=2]{latex'}}},
                      postaction={decorate}}}
\tikzset{<->/.style = {decoration={markings,
                                   mark=at position 0 with {\arrowreversed[scale=2]{latex'}},
                                   mark=at position 1 with {\arrow[scale=2]{latex'}}},
                       postaction={decorate}}}
\tikzset{->-/.style = {decoration={markings,
                                   mark=at position #1 with {\arrow[scale=2]{latex'}}},
                       postaction={decorate}}}
\tikzset{-<-/.style = {decoration={markings,
                                   mark=at position #1 with {\arrowreversed[scale=2]{latex'}}},
                       postaction={decorate}}}
\tikzset{->>/.style = {decoration={markings,
                                  mark=at position 1 with {\arrow[scale=2]{latex'}}},
                      postaction={decorate}}}
\tikzset{<<-/.style = {decoration={markings,
                                  mark=at position 0 with {\arrowreversed[scale=2]{twolatex'}}},
                      postaction={decorate}}}
\tikzset{<<->>/.style = {decoration={markings,
                                   mark=at position 0 with {\arrowreversed[scale=2]{twolatex'}},
                                   mark=at position 1 with {\arrow[scale=2]{twolatex'}}},
                       postaction={decorate}}}
\tikzset{->>-/.style = {decoration={markings,
                                   mark=at position #1 with {\arrow[scale=2]{twolatex'}}},
                       postaction={decorate}}}
\tikzset{-<<-/.style = {decoration={markings,
                                   mark=at position #1 with {\arrowreversed[scale=2]{twolatex'}}},
                       postaction={decorate}}}

\tikzset{circ/.style = {fill, circle, inner sep = 0, minimum size = 3}}
\tikzset{scirc/.style = {fill, circle, inner sep = 0, minimum size = 1.5}}
\tikzset{mstate/.style={circle, draw, blue, text=black, minimum width=0.7cm}}

\tikzset{eqpic/.style={baseline={([yshift=-.5ex]current bounding box.center)}}}
\tikzset{commutative diagrams/.cd,cdmap/.style={/tikz/column 1/.append style={anchor=base east},/tikz/column 2/.append style={anchor=base west},row sep=tiny}}

\definecolor{mblue}{rgb}{0.2, 0.3, 0.8}
\definecolor{morange}{rgb}{1, 0.5, 0}
\definecolor{mgreen}{rgb}{0.1, 0.4, 0.2}
\definecolor{mred}{rgb}{0.5, 0, 0}

\def\drawcirculararc(#1,#2)(#3,#4)(#5,#6){%
    \pgfmathsetmacro\cA{(#1*#1+#2*#2-#3*#3-#4*#4)/2}%
    \pgfmathsetmacro\cB{(#1*#1+#2*#2-#5*#5-#6*#6)/2}%
    \pgfmathsetmacro\cy{(\cB*(#1-#3)-\cA*(#1-#5))/%
                        ((#2-#6)*(#1-#3)-(#2-#4)*(#1-#5))}%
    \pgfmathsetmacro\cx{(\cA-\cy*(#2-#4))/(#1-#3)}%
    \pgfmathsetmacro\cr{sqrt((#1-\cx)*(#1-\cx)+(#2-\cy)*(#2-\cy))}%
    \pgfmathsetmacro\cA{atan2(#2-\cy,#1-\cx)}%
    \pgfmathsetmacro\cB{atan2(#6-\cy,#5-\cx)}%
    \pgfmathparse{\cB<\cA}%
    \ifnum\pgfmathresult=1
        \pgfmathsetmacro\cB{\cB+360}%
    \fi
    \draw (#1,#2) arc (\cA:\cB:\cr);%
}

\pgfarrowsdeclarecombine{twolatex'}{twolatex'}{latex'}{latex'}{latex'}{latex'}
\tikzset{->/.style = {decoration={markings,
                                  mark=at position 1 with {\arrow[scale=2]{latex'}}},
                      postaction={decorate}}}
\tikzset{<-/.style = {decoration={markings,
                                  mark=at position 0 with {\arrowreversed[scale=2]{latex'}}},
                      postaction={decorate}}}
\tikzset{<->/.style = {decoration={markings,
                                   mark=at position 0 with {\arrowreversed[scale=2]{latex'}},
                                   mark=at position 1 with {\arrow[scale=2]{latex'}}},
                       postaction={decorate}}}
\tikzset{->-/.style = {decoration={markings,
                                   mark=at position #1 with {\arrow[scale=2]{latex'}}},
                       postaction={decorate}}}
\tikzset{-<-/.style = {decoration={markings,
                                   mark=at position #1 with {\arrowreversed[scale=2]{latex'}}},
                       postaction={decorate}}}
\tikzset{->>/.style = {decoration={markings,
                                  mark=at position 1 with {\arrow[scale=2]{latex'}}},
                      postaction={decorate}}}
\tikzset{<<-/.style = {decoration={markings,
                                  mark=at position 0 with {\arrowreversed[scale=2]{twolatex'}}},
                      postaction={decorate}}}
\tikzset{<<->>/.style = {decoration={markings,
                                   mark=at position 0 with {\arrowreversed[scale=2]{twolatex'}},
                                   mark=at position 1 with {\arrow[scale=2]{twolatex'}}},
                       postaction={decorate}}}
\tikzset{->>-/.style = {decoration={markings,
                                   mark=at position #1 with {\arrow[scale=2]{twolatex'}}},
                       postaction={decorate}}}
\tikzset{-<<-/.style = {decoration={markings,
                                   mark=at position #1 with {\arrowreversed[scale=2]{twolatex'}}},
                       postaction={decorate}}}

\tikzset{circ/.style = {fill, circle, inner sep = 0, minimum size = 3}}
\tikzset{scirc/.style = {fill, circle, inner sep = 0, minimum size = 1.5}}
\tikzset{mstate/.style={circle, draw, blue, text=black, minimum width=0.7cm}}

\tikzset{eqpic/.style={baseline={([yshift=-.5ex]current bounding box.center)}}}
\tikzset{commutative diagrams/.cd,cdmap/.style={/tikz/column 1/.append style={anchor=base east},/tikz/column 2/.append style={anchor=base west},row sep=tiny}}

\definecolor{mblue}{rgb}{0.2, 0.3, 0.8}
\definecolor{morange}{rgb}{1, 0.5, 0}
\definecolor{mgreen}{rgb}{0.1, 0.4, 0.2}
\definecolor{mred}{rgb}{0.5, 0, 0}

\def\drawcirculararc(#1,#2)(#3,#4)(#5,#6){%
    \pgfmathsetmacro\cA{(#1*#1+#2*#2-#3*#3-#4*#4)/2}%
    \pgfmathsetmacro\cB{(#1*#1+#2*#2-#5*#5-#6*#6)/2}%
    \pgfmathsetmacro\cy{(\cB*(#1-#3)-\cA*(#1-#5))/%
                        ((#2-#6)*(#1-#3)-(#2-#4)*(#1-#5))}%
    \pgfmathsetmacro\cx{(\cA-\cy*(#2-#4))/(#1-#3)}%
    \pgfmathsetmacro\cr{sqrt((#1-\cx)*(#1-\cx)+(#2-\cy)*(#2-\cy))}%
    \pgfmathsetmacro\cA{atan2(#2-\cy,#1-\cx)}%
    \pgfmathsetmacro\cB{atan2(#6-\cy,#5-\cx)}%
    \pgfmathparse{\cB<\cA}%
    \ifnum\pgfmathresult=1
        \pgfmathsetmacro\cB{\cB+360}%
    \fi
    \draw (#1,#2) arc (\cA:\cB:\cr);%
}

\def\XXint#1#2#3{{\setbox0=\hbox{$#1{#2#3}{\int}$}
     \vcenter{\hbox{$#2#3$}}\kern-.5\wd0}}

\mathchardef\mdash="2D

\DeclareRobustCommand{\rvdots}{%
  \vbox{
    \baselineskip4\p@\lineskiplimit\z@
    \kern-\p@
    \hbox{.}\hbox{.}\hbox{.}
  }}

\usepackage{scalerel,stackengine}
\stackMath
\newcommand\reallywidehat[1]{%
\savestack{\tmpbox}{\stretchto{%
  \scaleto{%
    \scalerel*[\widthof{\ensuremath{#1}}]{\kern-.6pt\bigwedge\kern-.6pt}%
    {\rule[-\textheight/2]{1ex}{\textheight}}
  }{\textheight}%
}{0.5ex}}%
\stackon[1pt]{#1}{\tmpbox}%
}

\newtheorem{theorem}{Theorem}
\newtheorem*{theorem*}{Theorem}
\newtheorem{lemma}[theorem]{Lemma}
\newtheorem{proposition}[theorem]{Proposition}

\newtheorem*{corollary*}{Corollary}
\newtheorem{conjecture}[theorem]{Conjecture}

\theoremstyle{definition}

\newtheorem{definition}[theorem]{Definition}

\newtheorem{claim}[theorem]{Claim}

\theoremstyle{remark}

\newtheorem{remark}[theorem]{Remark}
\newtheorem*{remark*}{Remark}

\usepackage{lmodern,url,enumerate,mathtools,microtype}
\usepackage[hmargin=1in,vmargin=1in]{geometry}
\usepackage{graphicx}

\allowdisplaybreaks[1]

\definecolor{ForestGreen}{RGB}{34,139,34}

\theoremstyle{definition}
\newtheorem*{definition*}{Definition}
\newtheorem*{claim*}{Claim}

\newcommand{\ve}{\varepsilon}

\newcommand{\mr}[1]{{\rm #1}}

\newcommand{\la}{\langle}
\newcommand{\rg}{\rangle}

\usepackage{wrapfig,caption}
\newcommand{\cA}{\mathcal{A}}\newcommand{\cB}{\mathcal{B}}
\newcommand{\cC}{\mathcal{C}}\newcommand{\cD}{\mathcal{D}}

\newcommand{\cP}{\mathcal{P}}

\newcommand{\bR}{\mathbb{R}}
\newcommand{\bS}{\mathbb{S}}

\newcommand{\bZ}{\mathbb{Z}}

\newcommand{\pH}{\pazocal{H}}
\newcommand{\pP}{\pazocal{P}}
\newcommand{\pR}{\pazocal{R}}

\newcommand{\nc}{\newcommand}
\nc{\on}{\operatorname}
\nc{\Spec}{\on{Spec}}

\nc{\sn}{\mr{sn}}
\nc{\cn}{\mr{cn}}
\nc{\dn}{\mr{dn}}
\newcommand{\sfint}{\text{\sffamily{Int}}}
\newcommand{\sfgr}{\text{\sffamily{graph}}}

\numberwithin{equation}{section}
\makeatletter
\@addtoreset{equation}{section}
\makeatother

\title{Uniqueness of Semigraphical Translators} 

\author[F. Mart\'{\i}n]{\textsc{F. Mart\'{\i}n}}

\address{Francisco Mart\'in\newline
Departamento de Geometr\'{\i}a y Topolog\'{\i}a  \newline
Instituto de Matem\'aticas  de Granada (IMAG) \newline
Universidad de Granada\newline
18071 Granada, Spain\newline
{\sl E-mail address:} {\bf fmartin@ugr.es}
}%
\thanks{F. Mart\'{\i}n  was partially supported by the MICINN grant PID2020-116126-I00 and by the IMAG--Maria de Maeztu grant CEX2020-001105-M / AEI / 10.13039/501100011033.}
\author[M. S\'aez]{\textsc{M. S\'aez}}

\address{Mariel S\'aez\newline
P. Universidad Cat\'olica de Chile \newline
Facultad de Matem\'aticas\newline
Avda. Vicu\~na Mackenna 4860 \newline
Macul, Santiago, 6904441, Chile \newline
{\sl E-mail address:} {\bf mariel@mat.puc.cl}
}
\thanks{ M. S\'aez is partially supported by the grant Fondecyt Regular 1190388. }

\author[R. Tsiamis]{\textsc{R. Tsiamis}}

\address{Raphael Tsiamis\newline
Department of Mathematics \newline
Columbia University \newline 
2990 Broadway, New York NY 10027, USA \newline
{\sl E-mail address:} {\bf r.tsiamis@columbia.edu}
}
\thanks{ R. Tsiamis was partially supported by the A.G. Leventis Foundation Scholarship. 
He acknowledges the hospitality of the University of Granada, where part of this work was carried out.}
\date{\today}
\begin{document}

\maketitle

\vspace*{0.15in}


\begin{abstract}

We prove a conjecture by Hoffman, White, and the first author regarding the uniqueness of pitchfork and helicoid translators of the mean curvature flow in $\bR^3$. 
We employ an arc-counting argument motivated by Morse-Rad\'o theory for translators and a rotational maximum principle. 
Applications to the classification of semigraphical translators in $\bR^3$ and their limits are discussed, strengthening compactness results of the first author with Hoffman-White and with Gama-M\o ller.

\end{abstract}

\section{Introduction}
The mean curvature flow (MCF) is a geometric flow that evolves  a surface in the direction of  its mean curvature vector at each point. Translating solitons (or translators, for short) in the context of mean curvature flow are special types of solutions to the flow equation that preserve their shape along the flow.
Translators are interesting because they provide insights into the  mean curvature flow, specifically its long-time behavior, the formation of singularities, and the geometry of the evolving surfaces.
Some important results related to translators  include their classification in certain settings, the study of their stability properties, and their relationship to other geometric flows and geometric inequalities.

A  surface \(M\) is a translator for the mean curvature flow if there exists a fixed (unit) vector \(\bv\) such that
\(
M+ t \, \bv
\)
is a mean curvature flow, whereby $\mathbf{v}$ is called the velocity of the flow.
This means that 
\[ \vec H= \bv ^\perp .\]
Up to a rigid motion of $\R^3$, we can assume that $\bv=- \ee_3$, so the scalar mean curvature satisfies the equation $H = \la \nu, - \mathbf{e}_3 \rg$, where $\nu$ is the normal vector to the surface.

If $M=\sfgr(u)$, the equation $H + \la \nu, \textbf{e}_3 \rg = 0$ becomes the quasilinear elliptic PDE:
\begin{equation}\label{eqn:translatorequationR3}
    F(u,Du, D^2 u ) = \left( \Delta u + u_{xx} u_y^2 - 2 u_{xy} u_x u_y + u_{yy} u_x^2 \right) + \left(1 + |Du|^2 \right) = 0.
\end{equation}
Translators are minimal surfaces for Ilmanen's conformally flat metric $g_{\text{\sffamily{trans}}}$ given in \cite{ilmanen_1994} by $e^{-z} \delta_{ij}$; equivalently, they solve the Euler-Lagrange equation for the weighted volume functional $\mr{Vol}_{-z}(\Sigma) = \int_{\Sigma} e^{-z} \, d \mr{Vol}_{\Sigma}$. See \cite{himw-survey} for a survey about translating solitons.

In \cite{HMW1}, the existence and uniqueness of doubly periodic (``Scherk'') and simply periodic (``Scherkenoid'') translator surfaces is established.
These arise as the complete, simply connected, properly embedded translators resulting from repeated Schwarz reflection, around a vertical axes of symmetry, of the graph of a function solving equation \eqref{eqn:translatorequationR3} on:
\begin{enumerate}[(i)]
    \item for Scherk translators, a parallelogram $\cP(\alpha, w,L)$ of base angle $\alpha \in (0,\pi)$, width $w$, and unique length $L(\alpha,w)$, with boundary values $\pm \infty$ on the pairs of opposite sides; 
    \item for Scherkenoid translators, the half-strip of boundary data that results from extending the horizontal sides of $\cP(\alpha,w,L)$ above into rays in the $(+1,0)-$direction.
\end{enumerate}
A sequential compactness theorem \cite[Theorem 10.1]{HMW1} shows that one can extract (subsequential) limits from sequences of these Scherk or Scherkenoid surfaces as the base angle, $\alpha,$ tends to $\pi$. The behavior of these subsequential limits changes depending on whether $w < \pi$ or $w \geq \pi$.
These are accordingly named \textit{helicoids} or \textit{pitchforks}, coming from Scherk or Scherkenoid translators, respectively.
They are \textit{semigraphical translators} in $\bR^3$, i.e., properly embedded surfaces that are graphs after removing a discrete collection of vertical lines; these are classified by Hoffman, White, and the first author in \cite{HMW2}.
See \hyperref[pitchfork]{Figure 1} and \hyperref[helicoid]{Figure 2} for example illustrations of these translators and their fundamental pieces.
One is naturally led to the study of semigraphical translators as the next step following the classification of graphical solitons in $\bR^3$, due to Spruck-Xiao \cite{spruck-xiao} for the rigidity of the \textit{bowl soliton} and Hoffman-Ilmanen-White and the first author in the general case \cites{graphs, graphs-correction}.

The limit definition of helicoids and pitchforks implies various geometric properties, such as their asymptotic behavior and Gaussian curvature, cf. \cite[\S 11,12]{HMW1}.
In that paper, Hoffman, White, and the first author conjecture the following:
\begin{conjecture}[\cite{HMW1}]
For a given width $w$, the corresponding semigraphical translator $\Sigma_w$ (helicoid for $w < \pi$, or pitchfork for $w \geqslant \pi$) is unique up to vertical translation.
\end{conjecture}
We prove this conjecture and discuss its applications to the existence and uniqueness of limits of translators and to the classification of semigraphical translators. 

The main arguments are based on the so-called Morse-Rad\'o theory for minimal surfaces \cite{morserado} and what we call {\em rotational maximum principle}, which uses some ideas developed by F. Chini \cite{chini}.
See also \cite{HMW3} for related results.
\begin{figure}[ht]\label{pitchfork}
{\includegraphics[width=11cm]{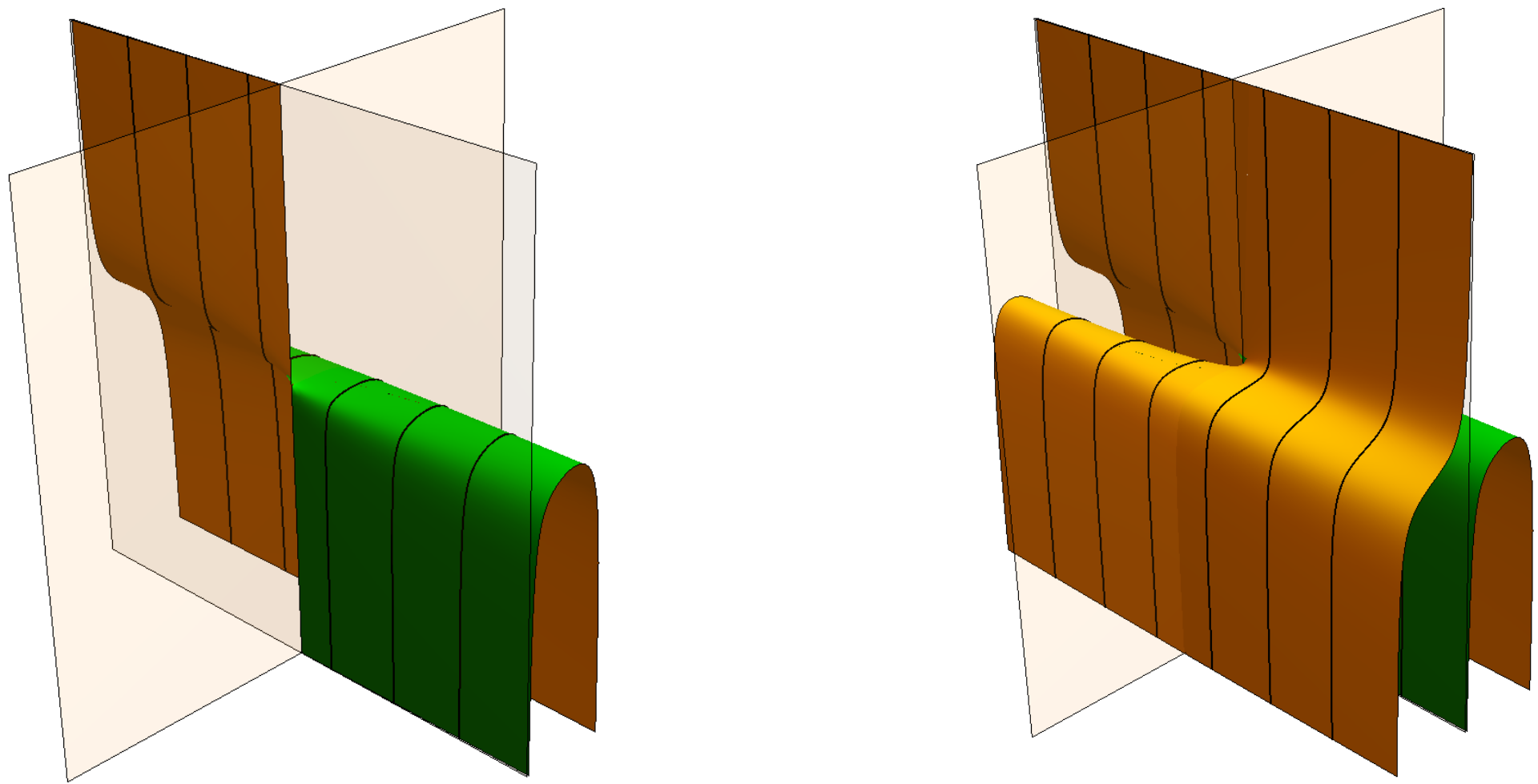}}
\caption{\small \sffamily From left to right: A fundamental piece of the pitchfork of width $\pi$ and the whole surface, obtained from a fundamental piece by a $180^{\rm o}$ rotation around the $z$-axis.
The asymptotic behavior at the two wings (vertical plane/ grim reaper) is visible here.}
\end{figure}

\subsection{Definitions and Preliminaries}

We recall the definitions of our translators from \cite[\S 11,12]{HMW1}.
The domain of definition will be the strip:
\begin{equation}\label{eqn:Strip}
    \Omega_w := \bR \times (0,w) = \left\{ (x,y) \in \bR^2 : 0 < y < w \right\}.
\end{equation}
We also define the cut-off half strips 
\begin{equation}\label{eqn:GeneralHalfStrips}
    \Omega^{> x_0}_w := \left\{ (x,y) \in \sfint(\Omega_w) : x > x_0 \right\}, \qquad 
    \Omega^{< x_0}_w := \left\{ (x,y) \in \sfint(\Omega_w) : x < x_0 \right\}.
\end{equation}

\begin{definition}[Pitchfork]\label{def:Pitchfork}
For any $w \geq \pi$, there exists a smooth translator $M$ whose boundary $\partial M$ is the $z-$axis $Z$ and whose interior $M \setminus \partial M$ is the graph of a function $u: \Omega_w \to \bR$ satisfying \eqref{eqn:translatorequationR3} on $\sfint(\Omega_w)$ with boundary values:
\begin{equation}\label{eqn:pitchforkboundaryvalues}
u(x,0) = \begin{cases}
    + \infty, & x < 0, \\ - \infty, & x > 0
\end{cases}, \quad \quad \text{and} \quad \quad u(x,w) = - \infty.  
\end{equation}
A \textit{pitchfork of width $w$} is the complete, simply connected translator without boundary $\pP_w$ obtained by performing a single Schwarz reflection of $M$ about the $z-$axis, $Z$. 
It follows that $\pP_w \setminus Z$ projects diffeomorphically onto $\{ - w < y < 0\} \cup \{0 < y < w\}$.
\end{definition}
\begin{definition}[Helicoid]\label{def:Helicoid}
For any $w < \pi$, there exists a smooth translator $M$ whose boundary $\partial M$ consists of two vertical lines, the $z-$axis and the line $\{ x = \hat{x}, y = w\}$ for some $\hat{x}>0$, and whose interior $M \setminus \partial M$ is the graph of a function $u: \Omega_w \to \bR$ satisfying \eqref{eqn:translatorequationR3} on $\sfint(\Omega_w)$ with boundary values:
\begin{equation}\label{eqn:helicoidboundaryvalues}
u(x,0) = \begin{cases}
    + \infty, & x < 0, \\ - \infty, & x > 0
\end{cases}, \quad \quad \text{and} \quad \quad u(x,w)  = \begin{cases}
    - \infty, & x < \hat{x}, \\ + \infty, & x > \hat{x}
\end{cases}.
\end{equation}
A \textit{helicoid of width $w$} is the complete, simply connected translator without boundary $\pH_w$ obtained from $M$ by performing countably many repeated Schwarz reflection about these axes.
It follows that $\pH_w$ contains the vertical lines $L_n$ through the points $n(\hat{x},w)$ for $n \in \bZ$ and $\pH_w \setminus \bigcup_n L_n$ projects diffeomorphically onto the strip cover $\cup_{n \in \bZ} \{ nw < y < (n+1) w\}$.
\end{definition}
In \cite[Theorem 11.1, Theorem 12.1]{HMW1}, the Gauss maps $\nu_{\Sigma}$ of the translators $\Sigma$ (helicoid $\pH_w$ or pitchfork $\pP_w$, respectively) are established to give a diffeomorphism from $\Sigma$ (in particular, from $M$) onto a region $\pR_w$ of the upper hemisphere in $\bS^2$.
The region $\pR_w$ depends only on $w$, and not on $\Sigma$; in fact, for a helicoid $\pH_w$, where $w < \pi$, the region $\pR_w$ is the entire upper hemisphere.

\begin{remark*}
Throughout the paper, we will use the notation $\Sigma$ or $\Sigma_w$ for a general complete semigraphical translator as above, which may be a helicoid or a pitchfork depending on whether $w< \pi$ or $w \geqslant \pi$.
When we want to make a distinction, we will adopt the notation introduced above: $\pH_w$ or $\pP_w$, respectively.
Our arguments will be primarily focused on the fundamental pieces of these translators, which will always be denoted by $M = \sfgr(u)$.
Following \eqref{eqn:GeneralHalfStrips}, we define the half-infinite parts of $M$,
\begin{equation}\label{eqn:GeneralHalfStripsgraph}
    M^{> x_0} := M \cap \{x>x_0\}, \qquad 
    M^{< x_0} := M \cap \{x<x_0\}.
\end{equation}
\end{remark*}

\begin{figure}[ht]\label{helicoid}
{\includegraphics[width=10cm]{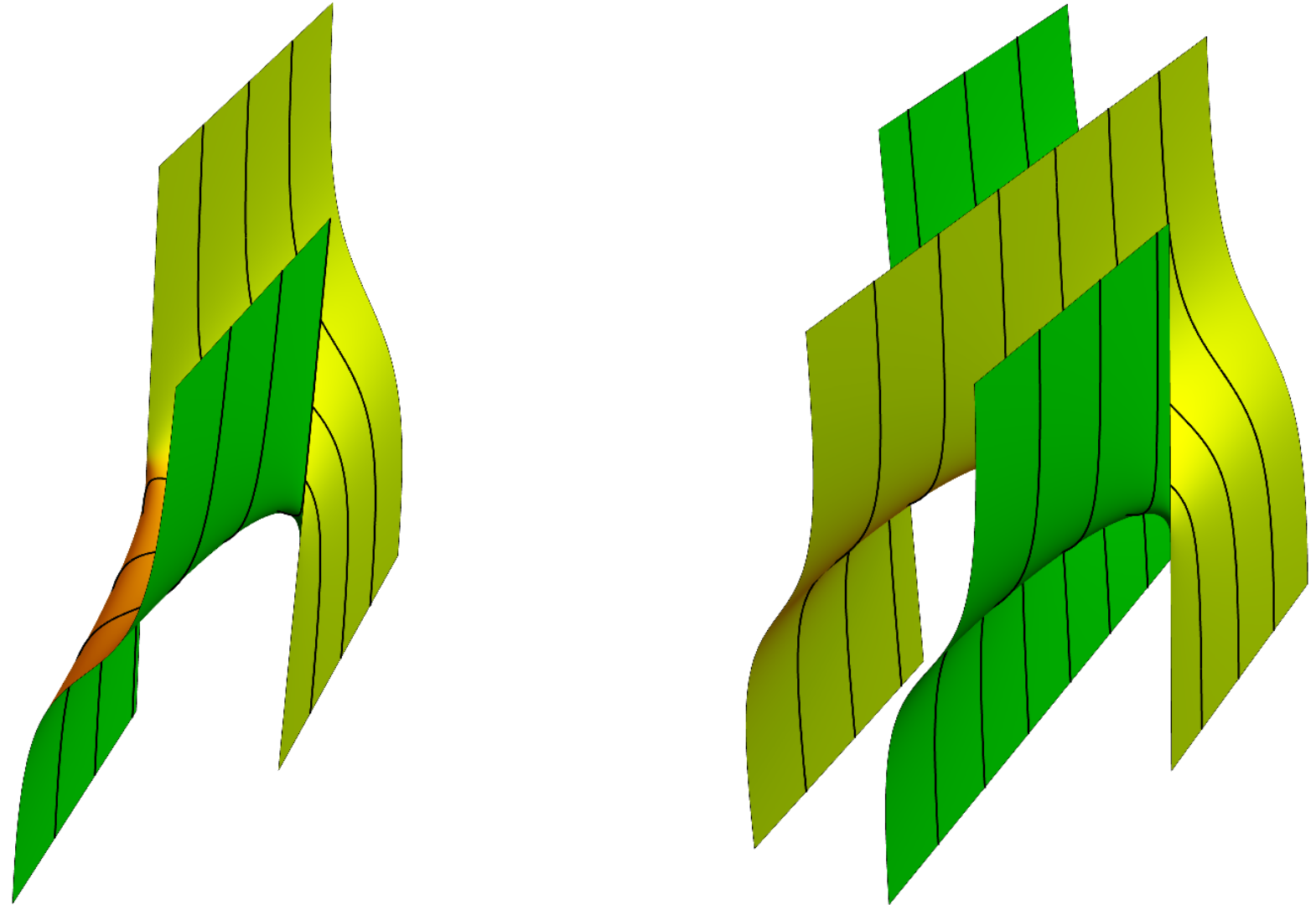}}
\caption{\small \sffamily From left to right: A fundamental piece of the helicoid of width $\pi/2$; and part of the surface, obtained by repeated reflection along vertical boundary lines.}
\end{figure}

\subsection{Acknowledgments}

The authors are thankful to Brian White, David Hoffman, and Eddygledson Souza Gama for many helpful discussions and comments.

\section{Uniqueness of Semigraphical Translators}

\subsection{Setup and proof technique} \label{subsec:setup}

We obtain the uniqueness (up to vertical shifts) of the semigraphical translators under consideration by an arc-counting argument that applies \cite[Proposition 4.2]{HMW1}.
Let $\Sigma_{1}$ and $\Sigma_2$ be two such translators of given width $w$, both pitchforks or both helicoids.
We study the functions $u_1$ and $u_2$ on the fundamental domain $\Omega_w$, which produce translators $M_1$ and $M_2$, which we call \textit{fundamental pieces}, that give rise to $\Sigma_{1}$  and $\Sigma_2$ upon the Schwarz reflections described above; these determine the resulting $\Sigma_i$ uniquely, for $i=1,2$.

Assuming by contradiction that the difference $u_1 - u_2$ is not a constant function then we have some point $p_0 \in \Omega_w$ such that $Du_1(p_0) \neq Du_2(p_0)$.
Since the Gauss map of each graph $u_i$, given by $\dfrac{(-Du, 1)}{\sqrt{1 + |Du|^2}}$, is a diffeomorphism of $M_i$ onto a region of the open upper hemisphere, we can produce point a $q_0 \neq p_0 \in \Omega_w$ such that $Du_1(q_0) = Du_2(p_0)$.
We then translate the domain by $\vec{\xi} = p_0 - q_0 = (\xi_1, \xi_2) \neq 0$ to obtain $\Omega'_w = \Omega_w + \vec{\xi}$, for which these points are now aligned.
Notice that for any vector $\vec{\xi} = (\xi_1, \xi_2)$ lying on the $(x,y)-$plane, the translation $T_{\vec{\xi}}(p) = p - \vec{\xi}$ on $\bR^3$ is an isometry of the Euclidean metric $\delta_{ij}$ and preserves the conformal factor $e^{-z}$, hence an isometry of $g_{\text{\sffamily{trans}}}$ preserving equation \eqref{eqn:translatorequationR3} and its solutions.

Similarly, we define $\Sigma'_1$ as the translator obtained by Schwarz reflection (as in Definition \ref{def:Pitchfork} or Definition \ref{def:Helicoid}, respectively) of the graph $M'_1 \setminus \partial M'_1 = \sfgr(u'_1: \Omega'_w \to \bR)$ of the function $u'_1(p) = u_1 (p - \vec{\xi} \, )$.
In the pitchfork case, $M'_1$ has boundary the $\vec{\xi}-$translated $z-$axis, $\{ x = \xi_1, y = \xi_2 \} \subset \bR^3$; in the helicoid case, the boundary $\partial M'_1$ consists of two vertical lines, $\vec{\xi}-$translated from Definition \ref{def:Helicoid}: $\{ x = \xi_1, y = \xi_2 \}$ and $\{ x = \hat{x} + \xi_1, y = w + \xi_2 \}$.
By construction,
\begin{equation}\label{eqn:ConstructionCriticalPoint}
D(u'_1 - u_2)|_{p_0} = Du_1(q_0) - Du_2(p_0) = 0
\end{equation}
so $p_0 \in \Omega_w \cap \Omega'_w$ is a critical point of the function $u'_1 - u_2$.
Notice that $u_1'$ and $u_2$ are both solutions of the $g_{\text{\sffamily{trans}}}$-minimal surface equation \eqref{eqn:translatorequationR3}, which is quasilinear, so their difference $ {v} =u'_1 - u_2$ solves a second-order linear elliptic equation:
\begin{equation}\label{pde-2}
   a^{ij} D_{ij} {v}  + b^i D_i {v} + c {v} = 0
\end{equation}
where $a^{ij}$, $b^i$, and $c$ are smooth functions of $(x,y)$.\footnote{See, for example, equation~(10.22) and the subsequent displayed formula in~\cite{GT}.
There, the derivation is for solutions of differential inequalities, but the same proof works when equality holds.}
Then, it is well-known that the critical points of a solution of \eqref{pde-2} are isolated; see, for instance, the survey paper \cite{magnanini}. 
{This shows that the function $u'_1 - u_2$ has isolated critical points if it is non-constant.}

\begin{lemma}\label{lemma:xi_2neq0}
For any open set $U \subset \Omega_w$, we can select a point $p_0 \in U$ such that the corresponding $q_0 \neq p_0$ and $\vec{\xi}  := p_0 - q_0 = (\xi_1, \xi_2)$ has $\xi_2 \neq 0$.
\end{lemma}
\begin{proof}
The point $q_0$ associated to $p_0$ is constructed above using the Gauss map diffeomorphisms, via the map $F := \nu_{\Sigma_1}^{-1} \circ \nu_{\Sigma_2}: \Omega_w \to \Omega_w$.  
Assuming for contradiction that $\xi_2 \equiv 0$ on $U$, we obtain a smooth map $g: U \to \bR$ such that $F|_U(p) = p + (g(p),0)$, so each connected component of $g(U)$ is a point or an interval.
For some $\eta \in \bR$, we denote (only here) $u^{\eta}_1(p) := u_1(p - (\eta,0))$.  
For any $\eta$, {if} $u_1^{\eta} - u_2$ {is constant} on a set of positive measure, {it follows by weak continuation that $u_1^{\eta} - u_2$ is constant on all of $\Omega_w$; then, looking at the boundary values forces $\eta = 0$ and $u_1 - u_2 = \text{\sffamily{const.}}$, contrary to our assumptions.}
Hence, $g(U) \subset \bR$ has positive measure, whereby a.e. non-empty preimages $\Gamma := g^{-1}(\eta) \subset U$ are $1-$dimensional submanifolds (by Sard's theorem). 
These would form non-isolated critical points of $u_1^{\eta} - u_2$, which is impossible, as we reasoned above.
\end{proof}
\begin{remark}[On $\vec{\xi} \neq 0$]\label{rmk:xineq0}
In what follows, we always consider $\vec{\xi} := p_0 - q_0 = (\xi_1, \xi_2)$ with $\xi_2 \neq 0$; this means that the common domain of definition $\Omega_w \cap \Omega'_w$ of $u'_1, u_2$ is strictly smaller than both $\Omega_w, \Omega'_w$; its boundary consists of one line from $\Omega_w$ and one from $\Omega'_w$.

This property is crucial, and the reason why our subsequent geometric arguments work.
The careful reader will notice that if the strips were allowed to have different widths $w' < w$, then we could have $\Omega'_{w'} \subset \sfint(\Omega_w)$, which would allow different configurations without producing a contradiction.
Indeed, there exist helicoids and pitchforks of all widths $w < \pi$ and $w \geqslant \pi$, respectively, that are pairwise distinct.

Throughout the paper, we will repeatedly use the fact that $u'_1$ and $u_2$ are solutions of the elliptic PDE \eqref{eqn:translatorequationR3}, so they are analytic. This means that if they coincide over an open subset, they would coincide over $\Omega_w \cap \Omega'_w$, by the unique continuation principle. 
For $\vec{\xi} \neq 0$, this is impossible: if $\xi_2 \neq 0$, then exactly one of $u'_1, u_2$ tends to $\pm \infty$ on each boundary line of $\Omega_w \cap \Omega'_w$, hence $u'_1 - u_2$ is non-constant; if $\xi_2 = 0$ and $\xi_1 \neq 0$, then $u'_1$ and $u_2$ have different $\pm \infty$ sign changes along the $x-$axis, so $u'_1 - u_2 \neq \text{\sffamily{const}}$.
\end{remark}
Since $p_0$ is an isolated critical point of $u'_1 - u_2$, we can study the arcs of the zero-level set $\{u'_1 - u_2 = 0\}$ through $p_0$.\footnote{The study of other level sets $\{u'_1 - u_2 = \lambda \}$ is equivalent, upon vertical translation of $\Sigma'_1 \cap \Sigma_2$, so we may without loss of generality we may assume that $u'_1(p_0) = u_2(p_0)$.} 
By \cite[Proposition 4.2]{HMW1}, proved using a Morse-theoretic argument, one deduces that the level set $\{ u'_1 - u_2 = 0\}$ consists of collections of an even number of arcs that pass through the critical point $p_0$ and suitable pairs of arcs assemble into analytic curves.
Since $p_0$ is a critical point, the level set through it is not a smooth submanifold locally at $p_0$, meaning that at least two curves intersect there, giving at least four arcs through $p_0$.
{In Lemma \ref{lemma:TypeOfArcs} below, we will sort these arcs into different types based on their endpoints at infinity.
Our key argument, using a rotational maximum principle, will demonstrate that there cannot exist more than one arc of each of these types.}

\subsection{Rotating translators}\label{subsect:RotatingTranslators}

For an angle $\theta > 0$, we denote by $R^{\theta}_{p_a}$ the rotation by $\theta$ with respect a point $p_a = (x_a, y_a,0)$, to be determined later, with axis parallel to the $z-$axis; restricted to the $(x,y)-$plane, this is $\begin{pmatrix}
            \cos \theta & - \sin \theta \\ 
            \sin \theta & \cos \theta
\end{pmatrix} \begin{pmatrix}
        x - x_a \\ y - y_a
        \end{pmatrix}$.
The image of an object under $R^{\theta}_{p_a}$ is denoted by the subscript $\theta$, so the images of the domain $\Omega'_w$, the translator $\Sigma'_1$, and functions $v: \Omega_w \to \bR, v': \Omega'_w \to \bR$
are, respectively, $\Omega'_{w,\theta}, \Sigma'_{1,\theta}$, and $v_{\theta}: \Omega_{w,\theta} \to \bR, v'_{\theta}: \Omega'_{w,\theta} \to \bR$.
The operation $R^{\theta}_{p_a}$ is an isometry of the metric $g_{\text{\sffamily{trans}}}$, hence, the translator equation \eqref{eqn:translatorequationR3} is invariant under it.
Thus, the rotated function $u'_{1,\theta}$ continues to solve equation \eqref{eqn:translatorequationR3} on the interior of $\Omega'_{w,\theta}$.
The region $\overline{\Omega_w \cap \Omega'_{w,\theta}}$ is compact: it is the parallelogram bounded between the lines $\{ y = 0, y=w \}$ and their images under $R^{\theta}_{p_a}$ composed with translation by $\vec{\xi}$.

The study of the function $u'_{1,\theta} - u_2$ on $\Omega_w \cap \Omega'_{w,\theta}$ amounts to examining the intersection of the translators $\Sigma'_{1,\theta}$ and $\Sigma_2$ along their fundamental domains.
If these have a well-behaved contact point there, the uniqueness of the translator follows from the tangency principle for translators, cf. \cite[Theorem 2.1]{MSHS15}, which uses the maximum principle to show that two translators must coincide if they have touching interiors and (locally, near the tangency point) lie entirely on one side of each other.
We will hence prove the uniqueness of the level set arcs of different types through $p_0$ as follows: we show that they would otherwise form an infinite sub-region of $\Omega_w \cap \Omega'_{w,\theta}$, over which the surfaces $M'_{1,\theta}, M_2$ can be rotated to satisfy the conditions of the tangency principle, with their contact occurring only along interior points. 
This would imply that the surfaces $M'_{1,\theta}, M_2$ coincide, due to the tangency principle, which contradicts the boundary conditions imposed on the functions $u'_{1,\theta}, u_2$.

To justify the argument presented above, we show that the interiors of these surfaces can be parametrized using the distance from the vertical axis through $p_a = (x_a, y_a, 0)$ and the $z-$coordinate; these, along with the azimuthal angle $\theta_{p_a}$, form the cylindrical coordinate system centered at $p_a$. 
We refer to a surface as a $\vartheta-$\textit{graph} if $\theta_{p_a}$ is expressible as a function of the other two parameters. 
More precisely, we use the following definition.
\begin{definition}[$\vartheta-$graph]
For a point $p_a = (x_a, y_a,0)$ on the $(x,y)-$plane, denote by $\rho_{p_a}(P) = \mr{dist}(P, \{ (x,y) = (x_a,y_a) \})$ the ambient Euclidean distance (in $\bR^3$) of a point $P$ from the vertical line through $p_a$, the axis of the cylindrical rotation.\footnote{To avoid confusion, here points of $\bR^3$ are denoted with uppercase letters $P,Q$ etc., while points of a planar domain $\Omega \subset \bR^2$ are denoted with lowercase letters $p,q,$ etc.}
We call an embedded surface $M \subset \bR^3$ a $\vartheta-$\textit{graph}, or $\vartheta-$\textit{graphical} over a domain $W \subset \bR_{\rho}^+ \times \bR_z$, if there is a point $p_a$ on the $(x,y)-$plane, outside $M$, such that:
\begin{enumerate}[(i)]
    \item $M$ is contained in a cylindrical sector of angle $\alpha < 2\pi$ centered at $p_a$;\footnote{A \textit{cylindrical sector} of angle $\alpha < 2\pi$ centered at $p_a$ is the interior of the half-space contained between two vertical half-planes passing through $p_a$ and intersecting at angle $\alpha$.} and 
    \item the ``cylindrical projection'' map to radius-height coordinates
\begin{eqnarray} \label{eqn:varphi-CoordinateMap}
\varphi_{p_a}|_M :  M &\longrightarrow &(0,\infty) \times \bR, \\ \varphi_{p_a}: \bR^3 \ni P &\longmapsto &\left( \rho_{p_a}(P), z(P) \right) \in [0,\infty) \times \bR \nonumber
\end{eqnarray}
is a diffeomorphism with image $W$.
\end{enumerate}
\noindent
Equivalently, the image of $M$ under the \textit{azimuthal angle function} $\theta_{p_a}$ is a graph $\vartheta$ over the region $W$ in the $(\rho_{p_a},z)-$plane. 
Here, $(\rho_{p_a}, \theta_{p_a}, z)$ are the cylindrical coordinates on $\bR^3$ centered at $p_a$,\footnote{The assumption that $M$ is contained in a cylindrical sector of angle $\alpha <2 \pi$ guarantees that the azimuthal angle $\theta_{p_a}$ is well-defined, up to translation by a constant, hence cylindrical coordinates are well-defined.}
 {so $\vartheta: W \to (0,\alpha)$ is given by $\vartheta(\varphi_{p_a}(P)) := \theta_{p_a}(P)$.}
\end{definition}
\begin{remark}
In this paper, all $\vartheta-$graph surfaces $M$ that we consider will be graphs (of a function, in the usual sense) over a region contained in a circular sector of angle $\alpha < \pi$, hence the first property \textbf{(i)} will be automatic.
\end{remark}
Notice that the preimage of a point $(s_1, s_2) \in \bR_{\rho}^+ \times \bR_z$ under the map $\varphi_{p_a}$ is given by the intersection of $M$ with the horizontal circle centered at $(x_a, y_a, s_2)$, with normal direction the vertical line and radius $s_1$; this is the set of points of the form 
\[ 
\varphi_{p_a}|_M^{-1}((s_1, s_2)) = M \cap \{ (x,y,s_2) : (x-x_a)^2 + (y-y_a)^2 = s_1^2 \}.
\]
It is clear that the surface $M$ is $\vartheta-$graphical for some $p_a$ if and only if so is its image under $R_{p_a}^{\alpha}$, for any $\alpha$.
Notice that the map $\varphi_{p_a}$ on $\bR^3$ is proper, since all the surfaces that we are considering are properly embedded.
 {
The $\vartheta-$graph properties are seen to imply that:
\begin{lemma}\label{lemma:SameImage}
Consider a simply connected domain $\Omega \subset \bR^2$ and functions $v_1, v_2: \Omega \to \bR$ such that both $V_i := \sfgr(v_i|_{\Omega})$ are $\vartheta-$graphs for the same point $p_a$.   
Assume that $V_1 \cap V_2$ is simply connected and $v_1 \equiv v_2$ on $\partial \Omega$.
Then, $V_1$ and $V_2$ have the same image $W$ under $\varphi_{p_a}$, hence are expressible as graphs $\vartheta_1, \vartheta_2$ over $W$ via their azimuthal angle $\theta_{p_a}$, and $\vartheta_1 \equiv \vartheta_2$ on $\partial W$. 
\end{lemma}
\begin{proof}
For a $\vartheta-$graph as $V_i = \sfgr(v_i|_{\Omega})$, the image $\varphi_{p_a}(V_i)$ is uniquely determined by $v_i|_{\partial \Omega}$.
This is because for any arc $\alpha_r \subset C_r(p_a) \cap \Omega$ of the circle of radius $r$ centered at $p_a$, $\rho_{p_a}|_{\alpha_r} \equiv r$ implies that $v_i|_{\alpha_r}$ is injective (because $\varphi_{p_a}|_{V_i}$ is injective), hence monotone and  $v_i(\alpha_r)$ is the interval with endpoints in $v_i(\partial \alpha_r)$.
For arc components of $\alpha_r = C_r(p_a) \cap \Omega$, we have $\partial \alpha_r = C_r(p_a) \cap \partial \Omega$.
From this, the
 surfaces $V_1, V_2$ have the same boundary and bound a simply connected region in $\bR^3$, so $v_1|_{\partial \Omega} = v_2|_{\partial \Omega}$ means that $\varphi_{p_a}(V_1) = \varphi_{p_a}(V_2) = W$.
By definition, $V_1, V_2$ are thus graphs $\vartheta_1, \vartheta_2$ over $W$, and $\vartheta_1|_{\partial W} = \vartheta_2|_{\partial W}$ due to $v_1|_{\partial \Omega} = v_2|_{\partial \Omega}$.
\end{proof}}
For pitchforks, being asymptotic to the vertical plane $y=w$ as $z \to - \infty$ will help us show that they are $\vartheta-$graphical, asymptotically in the $(-1,0)-$direction, along the slab $\{ 0 < y < w\}$.
For helicoids, the symmetry of their boundary data in both $(\pm 1,0)-$directions will enable us to establish this asymptotic property in both directions. 
More generally,
\begin{lemma}\label{Lemma:GraphicalImmersion}
Let $\Omega \subset \bR^2$ be a simply connected domain with boundary $\partial \Omega = C_1 \cup C_2$ and consider a smooth function $u: \Omega \to \bR$ whose graph is denoted by $M$.
Assume that $\rho_{p_a}(C_1) \cap \rho_{p_a}(\Omega) = \varnothing$ and $u|_{C_2} = \pm \infty$.
If the map $\varphi_{p_a}|_M$ is an immersion onto its image, then it is diffeomorphism (equivalently, $M$ is a $\vartheta-$graph).
This is equivalent to the vectors $(\partial_x u(p), \partial_y u(p))$ and $\vec{p} - \vec{p}_a$ never being collinear for $p \in \Omega$.
\end{lemma}
\begin{proof}
To prove the last equivalence, we precompose $\varphi_{p_a}|_M$ with the diffeomorphism $\Omega \ni p \mapsto (p,u(p)) \in M$ and we show that it is an immersion if and only if the map $\Omega \ni p \mapsto ( \rho_{p_a}(p), u(p)) \in \bR^+ \times \bR$ is an immersion as well.
The first entry then simplifies to $\rho_{p_a}(P) = \mr{dist}(P, \{ (x,y) = (x_a, y_a)\} = |p - p_a|$, which is smooth everywhere away from $p_a$, with differential $d \rho_{p_a} = \begin{pmatrix}
    \partial_x \rho_{p_a} \\ \partial_{y} \rho_{p_a}
\end{pmatrix}$, which is collinear with $\vec{p} - \vec{p}_a$. 
Since $d u(p) = (\partial_x u(p), \partial_y u(p))$, the composition with the map $\varphi_{p_a}|_M$ (hence, $\varphi_{p_a}|_M$ itself) is an immersion if and only if these are never collinear.
Then, $d (\varphi_{p_a}|_M)$ is everywhere bijective, so $\varphi_{p_a}|_M$ is an open map (invariance of domain), hence a local diffeomorphism.

Now we show the first assertion. We first prove that if $\varphi_{p_a}|_M$ is an immersion then it is a perfect map. 
To show that it has finite fibers, we proceed by contradiction: if some $(s_1, s_2) \in \bR^+ \times \bR$ had infinite fiber, its preimages would be an infinite subset of the horizontal circle of radius $s_1$ centered at $(p_a, s_2)$, which is compact, so there would be an accumulation point $(p_{\infty}, s_2)$ with $p_{\infty} \in \overline{\Omega}$ and a sequence $\{ (p_n, s_2 = u(p_n)) \} \subset \varphi_{p_a}|_M^{-1}((s_1, s_2))$ converging to this.
Then, $\rho_{p_a}(p_{\infty}) = s_1 \in \rho_{p_a}(\Omega)$ means $p_{\infty} \not\in C_1$ and $\lim_{p \to p_{\infty}} u(p) = s_2 < + \infty$ means $p_{\infty} \not\in C_2$, so $p_{\infty} \not\in \partial \Omega$ implies $p_{\infty} \in \Omega$ and $u(p_{\infty}) = s_2$, hence $(p_{\infty}, s_2) \in M$.
But then, $\varphi_{p_a}|_M$ would not be a local diffeomorphism at $(p_{\infty},s_2)$, since any neighborhood of the point contains distinct elements of $\{ (p_n, s_2) \}$, on which $\varphi_{p_a}|_M$ cannot be locally one-to-one; a contradiction.

To prove that the map $\varphi_{p_a}|_M: M \to \varphi_{p_a}(M)$ is closed, we consider a sequence $\{ (p_n, u(p_n) \} \subset M$; let its image under $\varphi_{p_a}$ have a limit point in $\varphi_{p_a}(M)$ denoted by $(\rho_{p_a}(q), u(q))$ for $q \in \Omega$.
Then, $|p_n - p_a| < |q-p_a| + \ve$ for a.e. $n$ means that the $\{p_n\}$ are contained in a compact subset of $\bR^2$, hence have an accumulation point $p_{\infty} \in \overline{\Omega}$ with $\varphi_{p_a}(p_{\infty}) = \varphi_{p_a}(q)$. 
By the same argument used above, $p_{\infty} \not\in \partial \Omega$, so $p_{\infty} \in  \overline{\{ p_n \}}_{\Omega} \subset \Omega$ means that $\varphi_{p_a}|_M$ is closed.

Since $\varphi_{p_a}|_M$ is a closed, surjective, continuous map with compact (finite) fibers, it is a perfect map, therefore proper.
It is also a local diffeomorphism, hence a finite-sheeted covering map.
$M = \sfgr(u: \Omega \to \bR)$ is homeomorphic to $\Omega$, thus simply connected, so the covering map is one-to-one and $\varphi_{p_a}|_M$ is a diffeomorphism onto $\varphi_{p_a}(M)$.
\end{proof}
This simple lemma helps to establish the validity of rotation arguments in various settings related to translators, e.g., in \cite[Proposition 6.3]{HMW1}.  
In fact, our next lemma shows that suitable regions of pitchforks and helicoids are $\vartheta$-graphs. 
More precisely, we have:
\begin{lemma}\label{lemma:PitchforkRotationalGraph}
For a pitchfork, there is a sufficiently large $R$ such that $M^{<-R}$ is a $\vartheta-$graph.
For a helicoid, there is a sufficiently large $R$ such both $M^{>R}$ and $M^{<-R}$ are $\vartheta-$graphs.
\end{lemma}
\begin{proof}
Let $\Omega^{\pm}_{p_a}(R)$ be the smaller region bounded between the rays $\{ y=w, y = 0\}$, both in the direction $(\pm 1,0)$, and the minor arc of the circle $C_R(p_a)$ of radius $R$ centered at $p_a$.
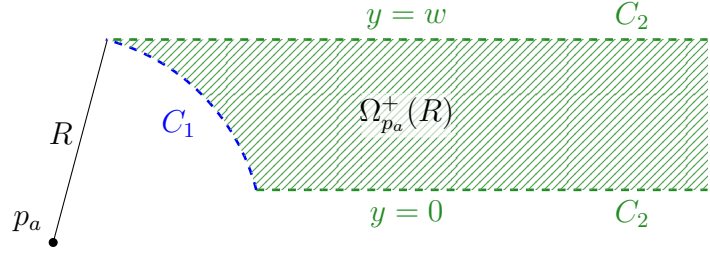
\begin{figure}[ht]
        \centering
\begin{tikzpicture}
    \draw[line width = 0.35mm, ForestGreen,dashed] (2.7,0.7) -- (8.7,0.7); 
    \draw[line width = 0.35mm, ForestGreen,dashed] (0.8,2.7) -- (8.7,2.7); 

    \draw (0,0) -- (0.72,2.7);

    \draw[name path=arc, line width = 0.35mm, blue, dashed] (2.7,0.7) arc
    [
        start angle=14.53,
        end angle= 75.47,
        radius=2.79cm
    ] ;

    \filldraw (0,0) circle (1.5pt) node[above left] {$p_a$};

    \node[ForestGreen, below] at (4.7,0.7) {$y = 0$};
    \node[ForestGreen, above] at (4.7,2.7) {$y = w$};

    \fill[pattern=north east lines, pattern color=ForestGreen!80] (2.7,0.7) rectangle (4.1,2.7);
    \fill[pattern=north east lines, pattern color=ForestGreen!80] (4.1,0.7) rectangle (5.3,1.4);
    \fill[pattern=north east lines, pattern color=ForestGreen!30] (4.1,1.4) rectangle (5.3,2);
    \fill[pattern=north east lines, pattern color=ForestGreen!80] (4.1,2) rectangle (5.3,2.7);
    \fill[pattern=north east lines, pattern color=ForestGreen!80] (5.3,0.7) rectangle (8.7,2.7);

    \node[ForestGreen, above] at (7.7,2.7) {$C_2$};
    \node[ForestGreen, below] at (7.7,0.7) {$C_2$};
    \node[blue, below] at (1.67,1.9) {$C_1$};
    \node at (0.1,1.45) {$R$};
    \node at (4.7, 1.7) {$\Omega_{p_a}^+(R)$};

    \fill[pattern=north east lines, pattern color=ForestGreen!80] (2.7,0.7) -- (2.7,0.7) arc [
        start angle=14.53,
        end angle=75.47,
        radius=2.79cm
    ] -- (2.7,2.7) -- cycle;
\end{tikzpicture}
\caption{ \sffamily The region $\Omega^+_{p_a}(R)$ extending in the $(+1,0)-$direction, used for a helicoid.}
        \end{figure}\label{fig:NewDomain}
\hyperref[fig:NewDomain]{Figure 3} depicts the region $\Omega^+_{p_a}(R)$ extending in the $(1,0)-$direction; for the region $\Omega^-_{p_a}(R)$ extending in the $(-1,0)-$direction, one reflects this with respect to the $y-$axis on the plane.
Letting $C_1$ be the circular arc in $\partial \Omega^{\pm}_{p_a}(R)$ and $C_2$ are the boundary rays  coming from $y=0$ and $y=w$, we see that the function $u$ solving the translator equation with boundary data \eqref{eqn:pitchforkboundaryvalues} or \eqref{eqn:helicoidboundaryvalues} has $u|_{C_2} = \pm \infty$, so the conditions of Lemma \ref{Lemma:GraphicalImmersion} are satisfied. 
For $\varphi_{p_a}$ to be a diffeomorphism when restricted to $\sfgr(u|_{\Omega^{\pm}_{p_a}(R)})$, it suffices to be an immersion.
The result will immediately follow from this by taking $R'$ such that $\Omega^{\pm}_{p_a}(R) \subset \Omega_w^{\gtrless \pm R'}$ and showing that under this condition $(\partial_x u(p), \partial_y u(p))$ and $\vec{p} - \vec{p}_a$ are not collinear.

First, note that if the derivatives $\partial_x u, \partial_y u$ vanish simultaneously, then the Gauss map $\nu$, given here by $p \mapsto \dfrac{(-\partial_x u, -\partial_y u, 1)}{\sqrt{1 + |\nabla u|^2}}$, equals $\mathbf{e}_3$ at that point. 
Since this is a diffeomorphism, there is at most one point $p$ with $\nu(p) = \mathbf{e}_3$, hence for $R \gg 0$ this cannot occur.
\begin{claim}
For a {\bf pitchfork}, there exists a $R_0 \gg 0$ and a $\ve > 0$ such that the slope bound
\begin{equation}\label{eqn:SlopeLowerBound}
\left| \dfrac{\partial_y u}{\partial_x u} \right| \geq \ve
\end{equation}
holds on the region $\Omega_w^{<-R_0}$.\footnote{This includes the possibility $\partial_x u = 0$, in which case the limit becomes $\pm \infty$.}
For a {\bf helicoid}, $R_0,\ve$ as above exist so that \eqref{eqn:SlopeLowerBound} holds on $|x| > R_0$.
\end{claim}
\begin{proof}
    Assume by contradiction that such $R_0, \ve$ do not exist; then, there is a sequence of points $\{ p_n\} \in \Omega_w$ with $x(p_n) \to - \infty$ for a pitchfork (or $\to \pm \infty$ for a helicoid) for which $\left| \frac{\partial_y u}{\partial_x u} \right| \to 0$.
    By result (3) of \cite[Theorem 12.1]{HMW1} for a pitchfork (respectively, result (1) of \cite[Theorem 11.1]{HMW1} for a helicoid), the sequence of translators $M_n := M - (p_n, u(p_n))$ converges smoothly to the vertical plane $\{ y = 0\}$, hence the Gauss map $\nu$ converges to $\pm \mathbf{e}_2$, meaning $\left| \frac{\partial_y u}{\partial_x u} \right| \to + \infty$; this contradicts the assumption $\left| \frac{\partial_y u}{\partial_x u} \right| \to 0$.
\end{proof}

In contrast to \eqref{eqn:SlopeLowerBound}, which holds for any $R \geq R_0$, we have on $\Omega^{\pm}_{p_a}(R)$,
\begin{equation}\label{eqn:paRBound}
\left| \dfrac{y - y_a}{x - x_a} \right| \leq \dfrac{|y| + |y_a|}{\sqrt{R^2 - (|y| + |y_a|)^2}} \leq \dfrac{w + |y_a|}{\sqrt{R^2 - (w + |y_a|)^2}} =: \delta(p_a,R).
\end{equation}
For $R \gg 0$ sufficiently large, the expression remains bounded by some (arbitrarily small) $\delta(p_a,R) < \ve$, so the vector $\vec{p} - \vec{p}_a$ remains close to the $x-$axis and not collinear with $(\partial_x u, \partial_y u)$.
By Lemma \ref{Lemma:GraphicalImmersion} it follows that $\varphi_{p_a}$ restricts to a diffeomorphism on $\sfgr(u|_{\Omega_{p_a}^{\pm}(R)})$.
\end{proof}
\begin{remark} The property of being $\vartheta$-graph for large $R$ does not hold in general for translators. 
For instance, subgraphs of grim reapers $G_w$ over $\Omega_w^{\lessgtr x_0}$ do not have it, because the critical curve $\cC_{p_a} := \left\{ p \in \Omega_w: d \varphi_{p_a}(p) = 0 \right\}$ cuts horizontally through the strip $\Omega_w$ and has components contained in $\Omega_w$ and extending to infinity, for any choice of $p_a$.
\end{remark}
The $\vartheta-$graph property extends to the intersection of the fundamental pieces $M'_{1,\theta} \cap M_2$, hence allowing us to use a common rotational coordinate for both; one reproduces the above argument, noting that the intersection is simply connected.
\begin{lemma}\label{lemma:SimplyConnectedIntersection}
For any simply connected open subset $U \subset \Omega_w$, the parts of the fundamental pieces $\sfgr(u'_{1,\theta}|_{U'_{\theta}}) \subset M'_{1,\theta}$ and $\sfgr(u_2|_U) \subset M_2$ have a simply connected intersection.
\end{lemma}
\begin{proof}
Since the $M_{i,U}$ are graphs over the $(x,y)-$plane, if there is a loop in $M'_{1, U'_{\theta}} \cap M_{2,U}$, 
projecting this to $\Omega_w$ would again give a closed loop (Jordan curve, bounding a compact set) $\gamma \Subset \Omega_w \cap \Omega'_{w,\theta}$ on which $u'_1 - u_2 = 0$, but
this is impossible by the maximum principle.
\end{proof}
\noindent
The defining property of $\vartheta-$graphs ensures that when the map $\varphi_{p_a}$ is a diffeomorphism, as in the case of the graphs of $u_i$ over $\Omega^{\pm}_{p_a}(R)$ by Lemma \ref{lemma:PitchforkRotationalGraph}, it takes interiors to interiors and therefore extends to the closure of these graphs inside $\Sigma_i$, taking the part of the boundary $\partial M_i$ to the boundary of $\text{\sffamily{im}}(\varphi_{p_a}) \subset \bR^2$.
Since the intersection of the translators $M'_1 \cap M_2$ is also simply connected, the same argument on proper local diffeomorphisms used in Lemma \ref{Lemma:GraphicalImmersion} implies that the interior of each $M'_1, M_2$ maps to the interior of $W$ and the boundary of each maps to the boundary of the image under $\varphi_{p_a}$.

\subsection{Uniqueness of arcs through $p_0$}
For our arc-counting argument by contradiction, we first classify the possible types of arcs of the level set $\{ u'_1 - u_2 = 0\}$ through the critical point $p_0$ and show that these are unique. 
We use the techniques and ideas of \cite{morserado}.

\begin{lemma}[Types of arcs of the level set through $p_0$]\label{lemma:TypeOfArcs}
The arcs of the zero-level set $\{ u'_1 - u_2 = 0\}$ passing through $p_0$ are contained in
\begin{itemize}
    \item $\sfint(\Omega_w \cap \Omega'_w) \cup \{ \vec{\xi} \, \}$, in the case of pitchfork translators \eqref{eqn:pitchforkboundaryvalues},
    \item $\sfint(\Omega_w \cap \Omega'_w) \cup \{ \vec{\xi} \, \} \cup \{ (\hat{x},w) \}$, in the case of helicoidal translators \eqref{eqn:helicoidboundaryvalues},
\end{itemize}
and have one of the following types:
\begin{enumerate}[(i)]
\item going to infinity in the $(1,0)-$direction;
\item going to infinity in the $(-1,0)-$direction;
\item passing through the point that is the projection of the vector $\vec{\xi}$ to the $(x,y)-$plane;\footnote{This point will also be denoted by  $\vec{\xi}$, for notational convenience.
This occurs in the left configuration of \hyperref[fig:HelicoidConfigs]{Figure 6} below, with $\xi_1 > 0$; see below for a remark on the configuration in case $\xi_2 < 0$.}
\item (only in the helicoidal case) passing through $(\hat{x},w)$.
\end{enumerate}
There exists precisely one arc of type (iii) and of type (iv) passing through $p_0$.
\end{lemma}
\begin{proof}
First, no closed arc of $\{ u'_1 - u_2  = 0\}$ contains $p_0$, else it would bound a set on which $u'_1 - u_2$ is non-zero on the interior, zero on the boundary, and satisfies a linear elliptic PDE coming from \eqref{eqn:translatorequationR3} with boundary value $0$.
These would mean that $u'_1 - u_2$ attains an interior maximum or minimum on this set, which is impossible by the maximum principle.
The containment in $\sfint(\Omega_w \cap \Omega'_w) \cup \{ \vec{\xi} \, \}$, possibly also $(\hat{x},w)$ in the helicoidal case, follows from the argument of Remark \ref{rmk:xineq0}, since exactly one of $u'_1, u_2$ becomes $\pm \infty$ along the boundary lines of $\Omega_w \cap \Omega'_w$, possibly with the exception of $\vec{\xi}$ and $(x,\hat{w})$, while the other remains finite, so these are the only possible intersection points of $\{ u'_1 - u_2 = 0 \}$ with the boundary lines.
Finally, every arc $\gamma: [0,+ \infty) \to \Omega_w \cap \Omega'_w$ not tending to these extends to infinity (proper): if it were bounded, it would have an accumulation point, hence a non-isolated zero of $u'_1 - u_2$; we have seen that this cannot happen.

In general, a unique arc of type (iii) or (iv) through $p_0$ arises for every boundary point of $M'_1 \cup M_2$ projecting to $\sfint(\Omega_w \cap \Omega'_w)$ such that the translator contains the entire vertical line through it. 
These are the points of $\Omega_w, \Omega'_w$ with a sign change of the (infinite) boundary data of the Dirichlet problem for the translator; they come in pairs $(p, p + \vec{\xi})$, of which exactly one is contained in $\Omega_w \cap \Omega'_w$ and uniquely determines the arc by $u'_1(p)$ or $u_2(p + \vec{\xi})$.\footnote{In the configuration of  \hyperref[fig:HelicoidConfigs]{Figure 6}, with $\xi_1 > 0$, these go through $\vec{\xi}$, and $(\hat{x},w)$ in the helicoidal case; for $\xi_2 < 0$, they would connect $p_0$ with $\vec{0}$ and $(\hat{x},w) + \vec{\xi}$.}
\end{proof}
We will also show the uniqueness of arcs of types (i), (ii), using the following setup: assume for contradiction that there are at least two arcs $\gamma_{1,2}: [0,+\infty) \to \Omega_w \cap \Omega'_w$ of the zero-level set $\{ u'_1 - u_2 = 0\}$ starting at $p_0$ and going to infinity in the same direction $(\pm 1,0)$, so these can be viewed as arcs on the original strips $\Omega_w \cap \Omega'_w$.
These cannot intersect again, else they would define between them a closed arc based at $p_0$, which was ruled out above. 
Hence,  without loss of generality we may  assume $\gamma_1$ remains above $\gamma_2$ as $t \to + \infty$, bounding a region $S_0 \subset \Omega_w \cap \Omega'_w$ between them that extends infinitely in the given direction.\footnote{We denoted $S_0$ as the region coming from the level set $\{ u'_1 - u_2 = 0 \}$; below, we will denote $S_{\lambda}$ as the corresponding region coming from $\{ u'_1 - u_2 = \lambda \}$, for arbitrary $\lambda$.}
The arc structure of the level set $\{ u'_1 - u_2 = 0\}$ may generally be quite complicated, for instance having nested regions and additional arc-crossings in addition to $p_0$. 
However, in every compact region the number of arcs within that region is finite, since otherwise we would have an accumulation point, which due to the analyticity of the solutions would imply that $u_1'\equiv u_2$.
\begin{lemma}[Evolution of the level set region]\label{Lemma:LevelSetEvolution}
Assume that at least two arcs $\gamma_i$ as above exist, forming a region $S_0$, and denote $L_1 := \inf_{S_0} (u'_1 - u_2), L_2 := \sup_{S_0} (u'_1 - u_2)$.
For a.e.$\lambda \in [L_1, L_2]$, the level set $\{u'_1 - u_2 = \lambda\}$ along $\Omega_w \cap \Omega'_w$ includes arcs as the $\gamma_i$ that bound a region $S_{\lambda} \subset S_0$ going to infinity in the $(\pm 1,0)-$direction.
For any $p_a, R$, there is a $\lambda \in [L_1,L_2]$ such that $S_{\lambda} \subset \Omega_{p_a}^{\pm}(R) \cap \Omega'_w$.
\end{lemma}
\begin{proof}
By Sard's theorem, a.e. $\lambda \in [L_1, L_2]$ is a regular value of $u'_1 - u_2$, so $(u'_1 - u_2)^{-1}(\lambda) \subset \Omega_w \cap \Omega'_w$ is a one-dimensional submanifold with at least one connected component contained in $S_0$.
The same argument as in Lemma \ref{lemma:xi_2neq0} shows that this must be a (collection of) proper open curve(s) diffeomorphic to $\bR$, hence having two ends extending to infinity in the $(\pm 1,0)-$direction.
Since the ends of the curve cannot intersect $\partial S_0$, they must extend in the same direction as $S_0$; these bound the desired semi-infinite region $S_{\lambda} \subset S_0$.

Since $u'_1 - u_2$ is non-constant, at least one of $|L_1|, |L_2| > 0$; by symmetry, we can assume that $L_2 > 0$ and consider an increasing sequence of regular values of $u'_1 - u_2$ with $\{ \lambda_i \} \to L_2$.
Iterating this argument for $(\lambda_i,\lambda_{i+1})$ in place of $(0,\lambda)$ gives a sequence of such regions $S_i := S_{\lambda_i}$ with $S_i \supset S_{i+1}$; clearly the boundaries $\partial S_i \cap \partial S_j = \varnothing$ for $i \neq j$.
We claim that one of these $\lambda_i$ has $S_i$ satisfying the desired property.
Otherwise, since already $S_i \subset \Omega_w \cap \Omega'_w$, suppose for contradiction that for all $i$ there is at least one $p_i \in (\Omega_{p_a}^{\pm}(R) \cap \Omega'_w)^C \cap S_i \neq \varnothing$.
The latter set is always contained in $( \Omega_{p_a}^{\pm}(R) \cap \Omega'_w )^C \cap S_0$, a compact region bounded between the lines $\{ y= c_i \}, \gamma_{1,2},$ and the circle $C_R(p_a)$.
By compactness, this infinite sequence $\{ p_i \}_{\lambda_i \to L_2}$ has an accumulation point $p_{\infty} \in \overline{S_0} \subset \sfint(\Omega_w \cap \Omega'_w)$. 
In fact, $p_{\infty} \not\in \partial S_0$ is an interior point of $S_0$, since $(u'_1 - u_2)(p_{\infty}) = L_2 > 0= (u'_1 - u_2)|_{\partial S_0}$.
Since $p_{\infty}$ is an interior point, the value $(u'_1 - u_2)(p_{\infty}) = \lim_{i \to \infty} (u'_1 - u_2)(p_i) = L_2$ is finite and $L_2 < + \infty$ is the finite maximum value of $u'_1 - u_2$, attained at $p_{\infty} \in \sfint(S_0)$, an interior point of the compact set $\Omega^{\pm}_{p_a}(R+1)^C \cap S_0$. 
This contradicts the maximum principle for $u'_1 - u_2$; thus, $S_{\lambda} \subset \Omega_{p_a}^{\pm}(R)$ for some $\lambda$.
\end{proof}
\begin{figure}[ht]
        \centering
        \includegraphics[scale = 0.6]{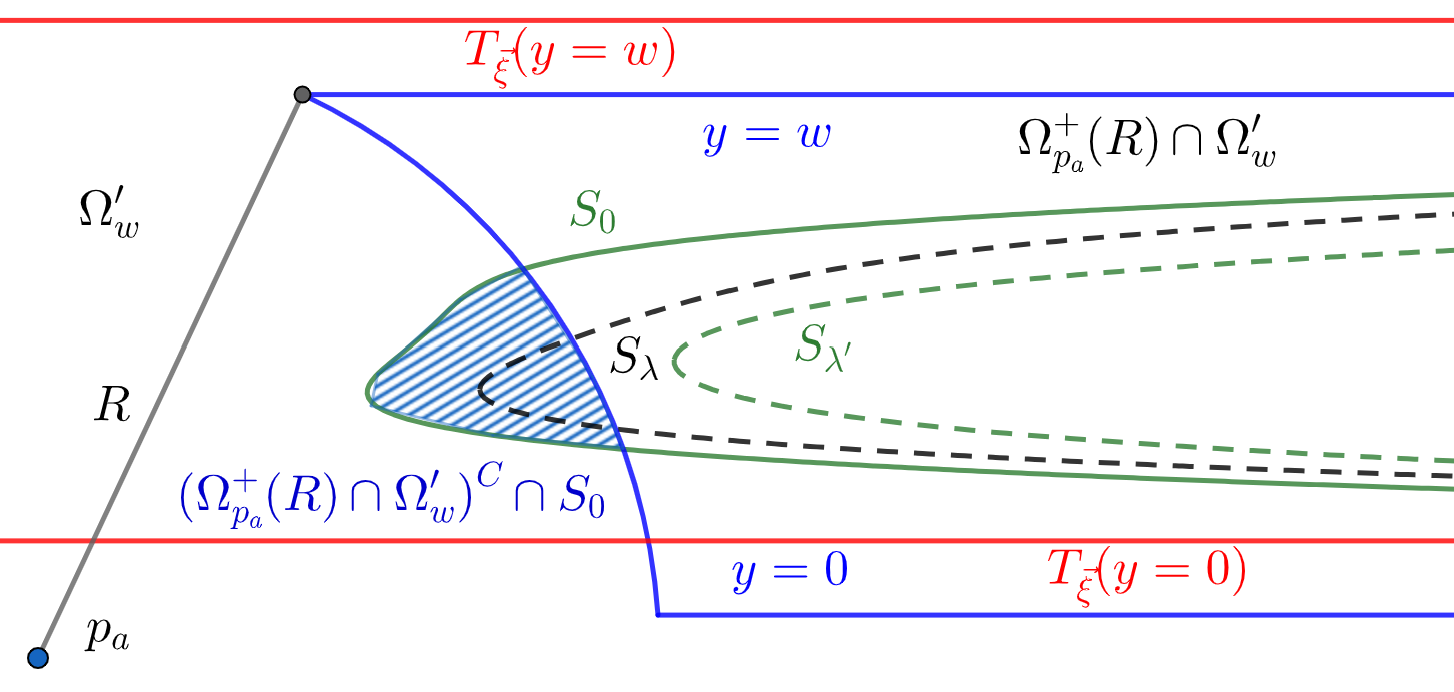}
        \caption{ \sffamily Evolution of the level set region.
        The strips $\Omega_w, \Omega'_w$ are between the blue and red lines respectively; $\Omega_{p_a}^+(R)$ is contained between the blue curves as before.
        The shaded region is the one used in the proof; the desired level set here is $S_{\lambda'}$.}
\end{figure}\label{fig:EvolutionLevelSets}
\begin{remark}
Note that we can pick the critical point $p_0$ to be contained in the region $\Omega^{\pm}_{p_a}(R) \cap \Omega'_w$ for any  given $p_a, R$: the slope bound \eqref{eqn:SlopeLowerBound} on $u$ gives an infinite region where $\left| \frac{\partial_y u_i}{\partial_x u_i} \right| > \ve$ for some $\ve$; by the diffeomorphism property of the Gauss map, this occurs on the complement of a compact set in $\Omega_w \cap \Omega'_w$ and is enough to ensure that the region $\Omega_{p_a}^{\pm}(R_i)$ is rotational for $u_i$ as long as $\delta(p_a, R_i) < \ve$ as in \eqref{eqn:paRBound}.
In the construction of $p_0, q_0, \vec{\xi}$ in \eqref{eqn:ConstructionCriticalPoint} via the Gauss map, we can choose $\tilde{R} = \max \{R_1, R_2\}$ so that $\Omega_{p_a}^{\pm}(\tilde{R}) \cap \Omega'_w$ is rotational for both $u'_1, u_2$, and take $p_0$ in it to guarantee that $q_0 = T_{\vec{\xi}}(p_0)$ remains in the rotational region. 
\end{remark}
Lemma \ref{Lemma:LevelSetEvolution} implies that we can always consider a translated level set (equivalently, a level set of $(M'_1 + \lambda \textbf{e}_3) \cap M_2$; this does not affect the critical point $p_0$) so that the region $S_{\lambda}$ is entirely contained in the rotational region $\Omega_{p_a}^{\pm}(R)$, within which the translator is a $\vartheta-$graph by Lemma \ref{lemma:PitchforkRotationalGraph}.
The only instance when this property is not accessible, requiring a different argument, is for a pitchfork with arcs extending in the $(+1,0)$ direction.
$M^{>R}$ is never a $\vartheta-$graph, because it is asymptotic to a grim reaper surface by \cite[Theorem 12.1]{HMW1}.
\begin{proposition}\label{prop:UniqueViaGrimReaper}
{In the above setting, if the two translators of the same width are pitchforks $\pP_1, \pP_2$, there exists a unique arc of the zero-level set passing through the critical point $p_0$ and going to infinity in the $(+1,0)$ direction.}
\end{proposition}
\begin{proof}
If there exist at least two such arcs, consider the semi-infinite region $S_0$ contained between them as in the above setup and take any sequence of points $\{ p_n \} \subset S_0$ with $x-$coordinates tending to $+\infty$. 

Due to White's compactness theorem, for any pitchfork $\pP$ constructed from a function $u$, the sequence of translators formed by
\[
\pP(n) := \pP - (p_n, u(p_n))
\]
has a smooth convergent subsequence in the uniform compact topology. 
For $x(p_n) \to + \infty$, by \cite[Theorem 12.1]{HMW1}, this limit is a tilted grim reaper $G_w$ over $\Omega_w$.
In particular, this is the case for $\pP_2$ and $\pP'_1$; these give rise to sequences $\pP_2(n)$ and $\pP'_1(n)$, with subsequences converging to $G_w$ and $G'_w$, respectively. 
Restricting to $\Omega_w$ and $\Omega'_w$, we get a graph $G_w$ coming from the tilted grim reaper function $g_w$,\footnote{Explicitly, this is given by $g_w(x,y) := x \sqrt{\left( \frac{w}{\pi} \right)^2 - 1} + \left( \frac{w}{\pi} \right)^2 \ln \left( \sin \left( \frac{\pi}{w} y \right) \right)$; however, we do not use this function here in any way besides its qualitative geometric properties.} while $G'_w$ comes from the graph of the translated $g'_w$ over $\Omega'_w$.
Since $\pP'_1, \pP_2$ have the same width $w$, they are asymptotic to grim reaper surfaces of the same slope in the $(+1,0)-$direction, defined over different strips $\Omega'_w, \Omega_w$.
We know from \cite[Theorem 6.7]{HMW1} that two (tilted) grim reapers diverge from each other at infinity, unless they meet at the origin, in which case they coincide. 
Since $\vec{\xi} \neq 0$ has $\xi_2 \neq 0$, the grim reapers do not coincide; hence $G_w \cap G'_w$ occurs along exactly one curve, so the region $S_0$ cannot exist and the arc through $p_0$ is unique.
\end{proof}
In the other cases, we follow the framework developed in Lemma \ref{Lemma:LevelSetEvolution} and Subsection \ref{subsect:RotatingTranslators}.
\begin{proposition}\label{prop:NumberOfArcsThroughp0}
In the setting of Lemma \ref{lemma:TypeOfArcs}, there is at most one arc of each type passing through $p_0$. 
\end{proposition}

\begin{proof}
The uniqueness of arcs of type (iii) and (iv) is already covered by Lemma \ref{lemma:TypeOfArcs}, while arcs of type (i) in the  case of pitchforks are unique by  Proposition \ref{prop:UniqueViaGrimReaper}.

 We now focus on all other situations that coincide with the setting of Lemma \ref{lemma:PitchforkRotationalGraph}, so $M'_1, M_2$ are $\vartheta-$graphs over some region $\Omega_{p_a}^{\pm}(R) \cap \Omega'_w = (\Omega_w \cap \Omega'_w)_{p_a}^{\pm}(R)$ for $R \gg 0$. 
 Here $\Omega_{p_a}^{\pm}(R) $ is defined as in the proof for Lemma \ref{lemma:PitchforkRotationalGraph} and it is depicted in \hyperref[fig:NewDomain]{Figure 3}. 

Assume for contradiction that at least two arcs of the same type exist; then, a region $S_0$ can be formed between them as above. 
By Lemma \ref{Lemma:LevelSetEvolution}, for some $\lambda$ we can evolve this into the semi-infinite region $\tilde{S} := S_{ \lambda} \subset \Omega_{p_a}^{\pm}(R) \cap \Omega'_w$.
Let
\begin{equation}\label{eqn:TranslatorParts-S}
    \tilde{M}'_1 := \sfgr(u'_1 +  \lambda)|_{\tilde{S}} \subset M'_1 +  \lambda \, \textbf{e}_3, \qquad \tilde{M}_2 := \sfgr(u_2)|_{\tilde{S}} \subset M_2.
\end{equation}
For ease of notation, we may work with the translated function $u'^{\lambda}_1 := u'_1 +  \lambda$ (simplified to just $u'_1$) so that the boundary curve $\partial \tilde{S} \subset \{ u'_1 - u_2 = 0\}$ is part of the zero-level set. 
\begin{claim}\label{claim:CompactIntersection}
Let $\tilde{M}'_{1,\theta} := R^{\theta}_{p_a}(\tilde{M}'_1)$. 
For any $|\theta| > 0$, the intersection $\tilde{M}'_{1,\theta} \cap \tilde{M}_2 \subset \bR^3$ is bounded (possibly empty); for an appropriate sign and $0 < \pm \theta \ll 1$ sufficiently small, it is non-empty.
For some $\theta = \theta_0$, the surfaces become tangent (the tangent planes coincide) along a set of points interior to both and $\tilde{M}'_{1,\theta_0}$ lies (locally) on one side of $\tilde{M}_2$ at some point of contact.
\end{claim}
\begin{proof}[Proof of Claim \ref{claim:CompactIntersection}]
Note that we are within the rotational region due to $\tilde{S} \subset \Omega_{p_a}^{\pm}(R) \cap \Omega'_w$, so $\tilde{M}'_1, \tilde{M}_2$ are $\vartheta-$graphs by Lemma \ref{lemma:PitchforkRotationalGraph}. 
Their intersection is simply connected, by Lemma \ref{lemma:SimplyConnectedIntersection}, hence Lemma \ref{lemma:SameImage} applies (for $\Omega = \tilde{S}, v_1 = u'_1 +  \lambda, v_2 = u_2$) to show that they can be expressed as $\vartheta$-graphs of functions $\vartheta_1, \vartheta_2$ over a region $W$, with $W= \text{\sffamily{im}} \, \varphi_{p_a} (\tilde{M}'_1) = \text{\sffamily{im}} \, \varphi_{p_a}(\tilde{M}_2) \subset \bR^+_{\rho} \times \bR_z$.
Its boundary $\partial W$ satisfies
\[ 
\partial W = \text{\sffamily{im}} \, \varphi_{p_a}(\{(p,u_2(p)) \; : \; p \in \partial \tilde{S}\}), \qquad \vartheta_1|_{\partial W} = \vartheta_2|_{\partial W}.
\]
The functions $\vartheta_1, \vartheta_2$ over $W$ represent the azimuthal angle function $\theta_{p_a}$ on $\tilde{M}'_1, \tilde{M}_2$ in the cylindrical coordinate system centered at $p_a$, whereby $|\vartheta_i| \leqslant \frac{1}{2}$ for $\lambda, R, |p_a| \gg 0$.\footnote{ {Since $\Omega_{p_a}^{\pm}(R)$ is contained in a circular sector of angle $\alpha < \pi$ centered at $p_a$, the azimuthal angle $\theta_{p_a}$ is well-defined up to a constant; hence, this just says that the oscillation of $\theta_{p_a}$ on $\Omega_{p_a}^{\pm}(R)$ is $\leq 1$ for $R, |p_a| \gg 0$.}} 
The map $R_{p_a}^{\theta}$ on $\tilde{M}'_1$ translates $\vartheta_1 \mapsto \vartheta_1 + \theta$ vertically and preserves the $\vartheta-$graph property.  
We denote $\tilde{S}_{\theta}=R_{p_a}^{\theta}(\tilde{S})$, then the intersections of the $\vartheta-$graphs and the surfaces $\tilde{M}$ are related by
\begin{align*}
(\varphi_{p_a})^{-1}(\vartheta_2 - (\vartheta_1 + \theta))^{-1}(0) &= \tilde{M}'_{1,\theta} \cap \tilde{M}_2, \\
\pi_z (\tilde{M}'_{1,\theta} \cap \tilde{M}_2) &= 
(u'_{1,\theta} - u_2)^{-1}(-  \lambda)  \subset \overline{\tilde{S} \cap \tilde{S}_{\theta}},
\end{align*} 
where $\pi_z$ denotes the projection from $\bR^3$ to the $(x,y)-$plane.
The latter set contains $\partial \tilde{S}$ for $\theta = 0$, since $\vartheta_1$ and $\vartheta_2$ agree on $\partial W$.
Thus, for an appropriate sign and $0 < \pm \theta \ll 1$ sufficiently small, the graphs of the functions $\vartheta_2$ and $(\vartheta_1 + \theta)$ have non-empty intersection as $\vartheta$-graphs over $\sfint(W)$.
Then, taking preimages as above gives $\tilde{M}'_{1,\theta} \cap \tilde{M}_2 \neq \varnothing$ for those $\theta$.

By Lemma \ref{lemma:TypeOfArcs}, $\partial \tilde{S}$ does not intersect $\partial (\Omega_w \cap \Omega'_w)$, hence $\overline{\tilde{S} \cap \tilde{S}_{\theta}} \Subset \Omega_w \cap \Omega'_{w,\theta}$ for these $\theta$; it is compact, since $\overline{\Omega_w \cap \Omega'_{w,\theta}}$ is also compact, and $u'_{1,\theta}, u_2$ are finite, hence continuous there.
It follows that $\overline{u_2^{-1}(\tilde{M}'_{1,\theta} \cap \tilde{M}_2})$ is compact, so $\tilde{M}'_{1,\theta} \cap \tilde{M}_2$ is bounded, as it has compact closure.

Depending on the choice of sign $+$ or $-$ for $\theta$ above, the desired angle $\theta_0$ is now $\sup_W (\vartheta_2 - \vartheta_1) > 0$ or $\inf_W(\vartheta_2 - \vartheta_1) < 0$, respectively;\footnote{One of these must be non-zero, else $\vartheta_1 = \vartheta_2$ would mean that $\tilde{M}'_1 = \tilde{M}_2$, whereby $u'_1 = u_2$ on $\tilde{S}$ would lead to the same contradiction obtained in the conclusion of this argument.} it remains to prove that this is attained on $W$.
Note that for any $\delta > 0$, the set $\bigcup_{\theta \in [\pm \delta,\pm \theta_0]} \tilde{S} \cap \tilde{S}_{\theta}$ is bounded; the above argument extends to show that $\bigcup_{\theta \in [\pm \delta, \pm \theta_0]} (\tilde{M}'_{1,\theta} \cap \tilde{M}_2)$ is  also bounded, thus, so is its image in $W$ under $\varphi_{p_a}$.
Conversely, $\tilde{S}$ having finite width $\leqslant w$ implies $\vartheta_2 - \vartheta_1 \to 0$ away from compact subsets,\footnote{This is clear to see by taking $p_a \in \sfint(\Omega_w \cap \Omega'_w)$ with $|p_a| \gg 0$, but also holds generally.} i.e.,
\[
\sup_{\substack{q_i \in \tilde{S}, \\ \rho_{p_a}(q_i) \geqslant r}} |\theta_{p_a}(q_1) - \theta_{p_a}(q_2)| \to 0, \quad \sup_{\substack{q_i \in \tilde{S}, \\ u(q_i) \geqslant r}} |\theta_{p_a}(q_1) - \theta_{p_a}(q_2)| \to 0, \qquad \text{as } \; r \to + \infty.
\]
so $\vartheta_2 - \vartheta_1 \to 0$ as $|\rho_{p_a}| \to + \infty$ or $|z| \to + \infty$.
This means that $\vartheta_2 - \vartheta_1$ attains the respective extremum $\theta_0$ in $\sfint(W)$. 
At (the preimages of) those points in $\tilde{M}'_{1,\theta_0} \cap \tilde{M}_2$, the gradients of the two surfaces align and they are (locally) the last points of contact, hence interior points of tangency in the manner prescribed above. 
This finishes the proof of Claim \ref{claim:CompactIntersection}.
\end{proof}
\noindent
Applying \cite[Theorem 2.1]{MSHS15} to the tangency of $\tilde{M}'_{1,\theta_0}, \tilde{M}_2$ shows that they coincide, meaning $u'_{1,\theta_0} + \lambda = u_2$ on $\tilde{S} \cap \tilde{S}_{\theta_0}$, thus on all of $\Omega_w \cap \Omega'_{w,\theta_0}$ by the weak continuation principle.
However, this is impossible as $u'_{1,\theta_0}$ is finite on the line of $\partial \Omega_w$ contained in $\Omega_w \cap \Omega'_{w,\theta_0}$, while $u_2 \to \pm \infty$.
Thus, a region $\tilde{S}$ as above cannot exist, so the arcs (i), (ii) are unique.
\end{proof}
\begin{figure}[ht]\label{rotation}
{\includegraphics[width=15cm]{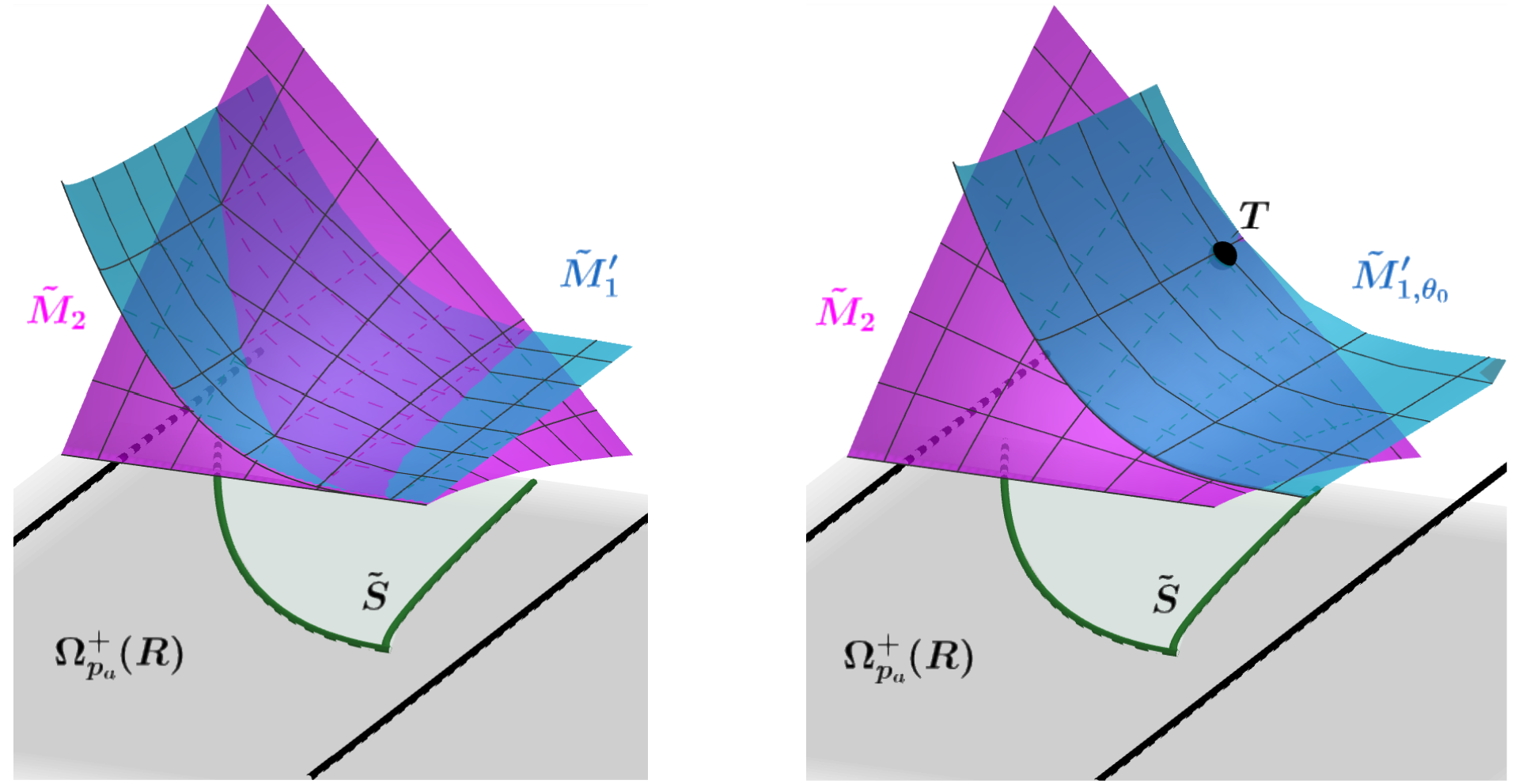}}
\caption{\small \sffamily Depiction of the process of rotating the surface $\tilde{M}'_1$ over the region $\tilde{S} \subset \Omega^+_{p_a}(R)$ (shaded), from the initial configuration (left) to the one obtained by rotating $\tilde{M}'_1$ through an angle $\theta_0$ (right) to produce the desired interior tangency with $\tilde{M}_2$ at $T$.}
\end{figure}

\subsection{Uniqueness of semigraphical translators}

Recall our setup: we consider two distinct semigraphical  solutions $u_1,\, u_2$ that are either both pitchforks or both  helicoids there is a critical point $p_0$ of $u_1(\cdot - \vec{\xi} \, )-u_2(\cdot)$ at which we assume (without loss of generality) that
$u'_1(p_0)=u_1 (p_0 - \vec{\xi} \, )=u_2(p_0)$. By
 \cite[Proposition 4.2 and Theorem 7.1]{HMW1}, the level set $\{ u'_1 - u_2 = 0\}$ consists of collections of an even number of arcs $(\geqslant 4)$ that pass through the critical point $p_0$ and suitable pairs of arcs assemble into analytic curves. 
They come in types (i)-(iv) by Lemma \ref{lemma:TypeOfArcs}; and Propositions \ref{prop:UniqueViaGrimReaper} and \ref{prop:NumberOfArcsThroughp0} show that the curves of each type are unique. 

We now prove uniqueness in each case assuming our setup and the notation introduced in previous sections.

\begin{theorem}[Uniqueness of pitchforks]\label{thm:UniquenessPitchforks}
Let $w \geq \pi$.
Then, there exists a unique (up to vertical translation) pitchfork translator $\pP_w \subset \bR^3$ of width $w$.
\end{theorem}
\begin{proof}
From Lemma \ref{lemma:TypeOfArcs} there is at most one arc of each type (i)-(iii) starting at $p_0$, hence there are at most 3 arcs. As previously stated, \cite[Proposition 4.2]{HMW1} implies that the level set $\{ u'_1 - u_2 = 0\}$ consists of collections of an even number of arcs that is at least $4$, which is a contradiction. 

We conclude that
$p_0$ cannot exist, which implies that $u_1 - u_2$ is constant and $\pP_w$ is unique up to translation.
\end{proof}
\begin{remark}[When $w = \pi$]
Recall from \cite[Theorem 10.1]{HMW1} that for $w=\pi$, the pitchfork $\pP_{\pi}$ arises from a possibly smaller connected component of the graph of the solution $u_{\pi}$, over a sub-domain 
$\tilde{\Omega} \subset \Omega_{\pi}$. 
This $\tilde{\Omega}$ still extends infinitely in both directions and has width $\pi$,\footnote{The \textit{width} of a domain of interest can be defined as $\sup_{p,q \in \Omega} |y(p) - y(q)|$.} so our considerations still apply: the translation of the domains by $\vec{\xi}$ does not produce a larger set of common vertices in $\Omega_{\pi} \cap \Omega'_{\pi}$ than in the general case $w > \pi$, hence the counting and uniqueness arguments for the arcs go through. 
Following the last step of the above argument, that $u'_{1,\theta_0} + \lambda = u_2$ on $\tilde{\Omega} \cap \tilde{\Omega}_{\theta_0}$, would again imply the equality on all of $\Omega_{\pi} \cap \Omega_{\pi,\theta_0}'$ by the uniqueness of weak continuation, which gives a contradiction as above.
\end{remark}

\noindent 
For helicoids, we will obtain a contradiction by setting up a configuration that produces two arcs extending in the same infinite direction.
We do this by comparing an arbitrary potential helicoid of width $w$ with the one arising as the translator limit of Scherk's type  translators.
\begin{theorem}[Uniqueness of helicoids]
Let $w < \pi$.
Then, there exists a unique (up to vertical translation) helicoid $\pH_w \subset \bR^3$ of width $w$.
\end{theorem}
\begin{proof}
    We will specifically prove that if $M_1=\sfgr(u_1)$ is an arbitrary helicoid of width $w$ and $M_2=\sfgr(u_2)$ is the helicoid of width $w$ constructed in \cite[Theorem 12.1]{HMW1}, then $\pH_1 = \pH_2$ up to vertical translation.
    The fundamental piece $M_2$ of $\pH_2$ is constructed as follows: for any $\alpha \in (0,\pi)$ and $w \in (0,\infty)$, there is a unique $L = L(\alpha,w)$ such that the boundary value problem for a translator $\cD = \sfgr( \tilde{u} : \cP(\alpha,w,L) \to \bR)$ has a (unique) solution $\tilde{u}$ taking values $\pm \infty$ on the alternating sides of the parallelogram $\cP(\alpha,w,L)$ of base angle $\alpha$, width (height) $w$, and horizontal base length $L$ with lower-left vertex at the origin. 
    By \cite[Theorem 9.7]{HMW1}, we obtain $M_2$ and its function $u_2$ as the smooth translator limit of the $\cD_i$ for a sequence of $\{ \alpha_i \} \to \pi$ and $\{ w_i \} \to w < \pi$, with $\{ \tilde{u}_i \} \to u_2$ smoothly, $\{ L_i\} \to + \infty$, and with smooth dependence of $L, \cD$ on $(\alpha,w)$.
    Also, the upper-right vertex of the $\cP(\alpha_i,w_i,L_i)$ converges to $(\hat{x}_2,w)$ for $\hat{x}_2(u_2)$ computed in \cite[Proposition 11.3]{HMW1}.

Hence, we proceed again by contradiction. We assume that $u_1-u_2$ is not constant and then we are able to find $p_0$ and $\vec \xi$ as in Subsection \ref{subsec:setup}, such that defining $u_1'(p):=u_1(p-\vec \xi)$ makes $p_0$ a critical point of $u'_1 - u_2$.
By Proposition \ref{prop:NumberOfArcsThroughp0}, the level set $\{ u'_1 - u_2 = 0\}$ of $M'_1 \cap M_2$ contains at most one arc of each type (i) - (iv) going through $p_0$; since we need at least $4$ due to $p_0$ being a critical point, there is exactly one arc of each type. 

Let us consider the functions
\[
F(x,y,z) := z - u_2(x,y), \qquad F_i := z - \tilde{u}_i(x,y).
\]
From our construction of $M_2$, we have that $F_i \to F$ uniformly.
Furthermore, we are assuming that the number of critical points (counting multiplicity) is $\geq 1$, due to $p_0$:
    \begin{equation} \label{eq:N1} \mathsf{N}(F_2 \,| \, M_1') \geq 1, \end{equation}
   where $\mathsf{N} ( F \, | \, M)$ denotes the number of critical points of a function $F$ on $M$.
     By the lower semicontinuity of $\mathsf{N}(\cdot)$, proved in \cite[Theorem 39, Corollary 40]{morserado}, we have that
     \begin{equation} \label{eq:Ni}         
     \liminf_i \mathsf{N}(F_i \, | \, M_1') \geq \mathsf{N}(F_2 \, | \, M_1') \geq 1. \end{equation} 
     Hence, the critical point $p_0$ of $u'_1 - u_2$ arises as the limit (after possibly passing to a subsequence) of critical points $p_i$ of the functions $u'_1 - \tilde{u}_i$; the level set curves through $p_0$ are likewise the $C^{\infty}$ limits of the level set curves through the $p_i$.
    For each $i$, the set $V_i$ of vertices from \cite[Proposition 4.2]{HMW1}\footnote{In the paper \cite{HMW1}, this set is denoted by $S$ for the intersection of the two domains.} for the domain $\Omega'_w \cap \cP(\alpha_i, w_i,L_i)$ consists of interior vertices and points $P_{\ell}$ where the $(+\infty, + \infty)$ or $(-\infty, -\infty)$ edges of the two domains intersect.
    We can focus on $i \gg 0$ sufficiently large above, so that:
    \begin{enumerate}[(i)]
        \item all $|w - w_i| < |\xi_2|$, since $\xi_2 \neq 0$; this ensures that the configuration of the set of vertices will be the one described below, and the set $V_i$ does not become smaller.\footnote{This would happen if we tried to reproduce the argument for two helicoids of different widths $w \neq w'$; these are known to be distinct, hence the argument would fail.} 
        \item all $\alpha_i \in (\pi - \delta, \pi)$ for some $\delta = \delta(w,\xi)$; this ensures that one of the two vertices $(0,0) + \vec{\xi}$ or $(\hat{x}_1,w) + \vec{\xi}$ of $\Omega'_w$ is contained in $\cP(\alpha_i,w_i,L_i)$, depending on $\text{\sffamily{sgn}}(\xi_2) \neq 0$.
    \end{enumerate} 
    Since $p_i$ is a critical point of $u'_1 - \tilde{u}_i$, there is an even number number $2r_i \geqslant 4$ of arcs joining it to the points $ \{ (P_{\ell})_i \}_{\ell=1}^4$ of $V_i$; these assemble, in pairs, to analytic curves through $p_i$.
    The index $i$ in $(-)_i$ comes from the $\cP_i(\alpha_i,w_i, L_i)$ and will be suppressed when clear, to ease notation.
    These are the only arc types appearing through $p_i$, as arcs cannot have interior accumulation points or self-intersection, as showed in Lemma \ref{lemma:TypeOfArcs}, and can only have boundary endpoints in $V_i$, because $u'_1 - \tilde{u}_i \to \pm \infty$ along the rest of $\partial \left(\Omega_w \cap \cP(\alpha_i,w_i,L_i)\right)$.

    Let $\{ (\gamma_{\ell})_i \}_{\ell=1}^4$ be the different types of arcs connecting $p_i$ to the points $\{ (P_{\ell})_i \}_{\ell = 1}^4$, depicted in the configurations of \hyperref[fig:HelicoidConfigs]{Figure 6}.
    Depending on their type, they each converge to some arc of type (i)-(iv) from Lemma \ref{lemma:TypeOfArcs} in the limit.\footnote{In what follows, we will use the phrase ``in the limit'' as shorthand for ``in the limiting configuration on $M'_1 \cap M_2$ coming from $M'_1 \cap \cP(\alpha_i,w_i,L_i)$ as $i \to \infty$.''}
    
    \begin{figure}[ht]
        \centering
        \includegraphics[scale = 0.32]{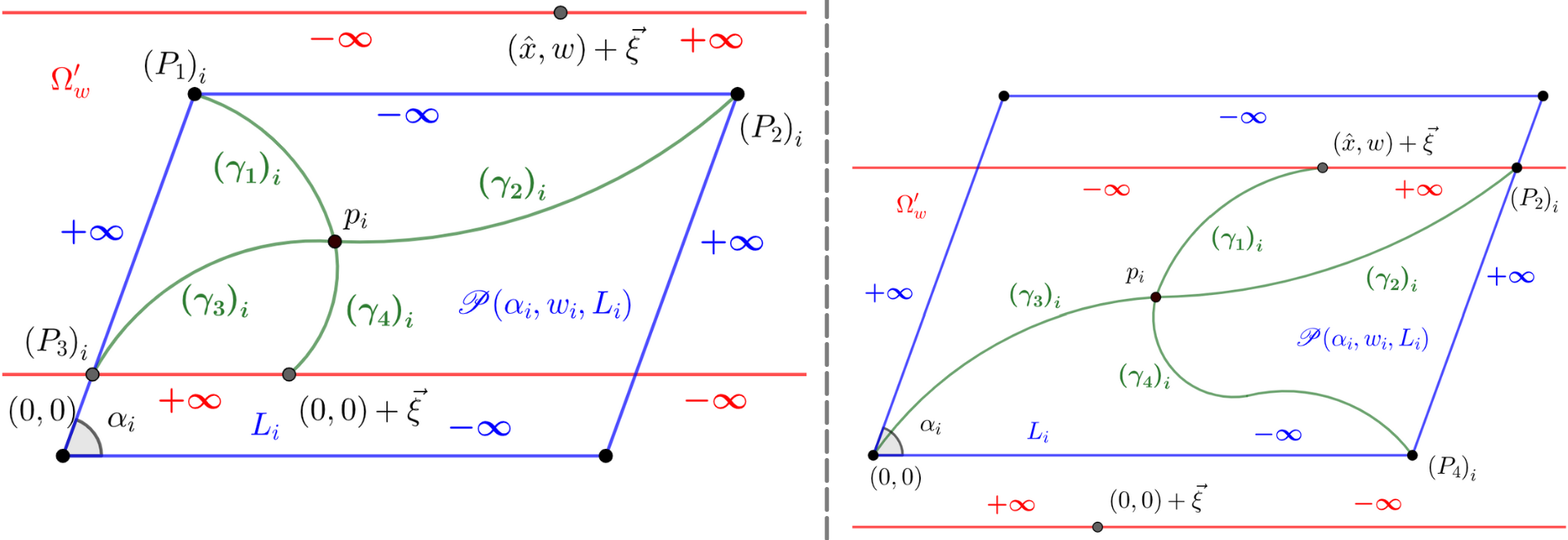}
        \caption{ \sffamily The two possible relative configurations and the resulting arcs, depending on $\xi_2 > 0$ (left) or $\xi_2 < 0$ (right).
        Here $\alpha_i < \pi/2$ for the convenience of the reader.}
        \end{figure}\label{fig:HelicoidConfigs}

    We focus without loss of generality on the first configuration \textbf{(a)}, where $\xi_2 > 0$; the arguments for the second configuration \textbf{(b)}, where $\xi_2 <0$, are analogous upon changing $(-1,0)-$arcs to $(+1,0)-$arcs (type (ii) to type (i)), and $\{ (0,0) + \vec{\xi}, (\hat{x},w) \}$ to $\{ (0,0), (\hat{x},w) + \vec{\xi} \}$.

We know, from \eqref{eq:Ni}, that 
\begin{equation}\label{eq:>4}\mbox{$\{u_1'-\tilde{u}_i=0\}$ consists of at least $4$ arcs meeting at $p_i$.} 
\end{equation}
Since each $\{u_1'- \tilde{u}_i=0\}$ is simply connected and the $\{ M'_1, \widetilde{M}_i \}$ are proper, these arcs can only end at a point in $V_i = \{ (P_1)_i, (P_2)_i, (P_3)_i, \vec \xi \}.$ 
It is clear that the arcs through $\{ P_1, P_2, \vec{\xi} \}$ exist and are unique, because one of $u_1$ or $\tilde{u}_i$ is finite in a neighborhood of these points and the maximum principle applies. 
The arc through $P_3$ is also unique: $\widetilde M_i $ converges, as $z \to \infty$, to two parallel vertical strips; the one appearing near $P_3$ is 
\[
s=\{ y= \tan (\alpha_i) x, \; 0 \leq y \leq w \}.
\]
    On the other hand, $M_1'$ converges, as $z\to \infty$, to the half-plane \[h=\{y=\xi_2, \; x \leq \xi_1\}.\]
    As $s \cap h=\{P_3\} \times \mathbb{R}$, $\{u_1'-\widetilde u_i=0\} $ in a neighborhood of $P_3$ contains only one open arc (ending at $P_3$) that corresponds to the projection of a curve in $M_1'\cap \widetilde M_i$ asymptotic to the line $\{P_3\} \times \mathbb{R}.$
    It follows that
\begin{equation} \label{eq:<4}
    \mbox{$\{u_1'-\tilde{u}_i=0\}$ consists of at most $4$ arcs meeting at $p_i$.}
\end{equation}
    Together, the expressions \eqref{eq:>4} and \eqref{eq:<4} imply that there exactly four arcs in $\{u_1'-\tilde{u}_i=0\}$ meeting at $p_i$.
In the notation of Figure \ref{fig:HelicoidConfigs}, we obtain:
    \[ 
    \{u_1'-\tilde{u}_i=0\}=(\gamma_1)_i \cup (\gamma_2)_i \cup (\gamma_3)_i \cup (\gamma_4)_i.
    \]
    Then $\mathsf{N}(F_i \, | \, M_1')=1$, hence \eqref{eq:N1} and \eqref{eq:Ni} give 
    $\mathsf{N}(F_2 \, | \, M_1')=1$. 
    This means that $\{ u_1'-u_2=0\}$ is the union of exactly 4 curves, which are limits of the curves $(\gamma_1)_i, \ldots, (\gamma_4)_i$. 
    Observe from the configuration of the vertices that $(P_1)_i \to (-\infty, w)$ and $(P_3)_i \to (-\infty, \xi_2)$, meaning that two of these arcs, the limits of $(\gamma_1)_i$ and $\gamma_3)_i$, are  going to infinity in the $(-1,0)-$direction; this contradicts Proposition \ref{prop:NumberOfArcsThroughp0}.

Thus, such a critical point $p_0$ cannot exist, so $u_1 - u_2$ is constant.
\end{proof}

\begin{remark}
In the configuration \textbf{(b)} (on the right) with $\xi_2 < 0$, we would instead have obtained the contradiction via the above argument by noticing that the curves of types $(\gamma_2)_i, (\gamma_4)_i$ converge to two distinct arcs of type (i) through $p_0$ in the limit, which again contradicts Proposition \ref{prop:NumberOfArcsThroughp0}.     \end{remark}

\section{Applications}

The uniqueness of pitchforks and helicoids allows us to complete previous results and questions of Hoffman-Mart\'in-White \cites{HMW1, HMW2} and Gama-Mart\'in-M\o ller \cite{GMM} in a satisfying manner. 
We stress that they treat the main steps of these theorems, establishing the compactness of certain sequences of translators; our contribution is the existence and uniqueness of the resulting limits in the strong sense (rather than subsequential). 

To state the first few results, recall the construction of \textit{translating tridents} due to Nguyen, \cite{nguyen}; see also \cite[Theorem 1]{HMW2}.
\begin{definition}
For every $a>0$, let the \textit{trident} $M_a$ \textit{of neck size $a$} be the unique properly embedded semigraphical translator in $\bR^3$ that
\begin{enumerate}[(i)]
    \item is periodic with period $(2a,0,0)$,
    \item contains the vertical line $\{ (na,0) \} \times \bR$, for each integer $n \in \bZ$,
    \item has fundamental piece (for Schwarz reflection) the graph of the solution $u_a$ of \eqref{eqn:translatorequationR3} over a strip $\bR \times (0,b)$, for a unique $b = b(a)$, with boundary data 
    \[
    u_a(x,0) = \begin{cases}
        - \infty, & - a < x < 0, \\
        + \infty, & 0 < x < a
    \end{cases}, \qquad u_a(x,b) = - \infty.
    \]
    \item is tangent to the $(y,z)-$plane at the origin.
\end{enumerate}   
\end{definition}
Two examples of tridents of different neck sizes are given in \hyperref[tridents]{Figure 7}.
\begin{figure}[ht]\label{tridents}
{\includegraphics[width=7cm]{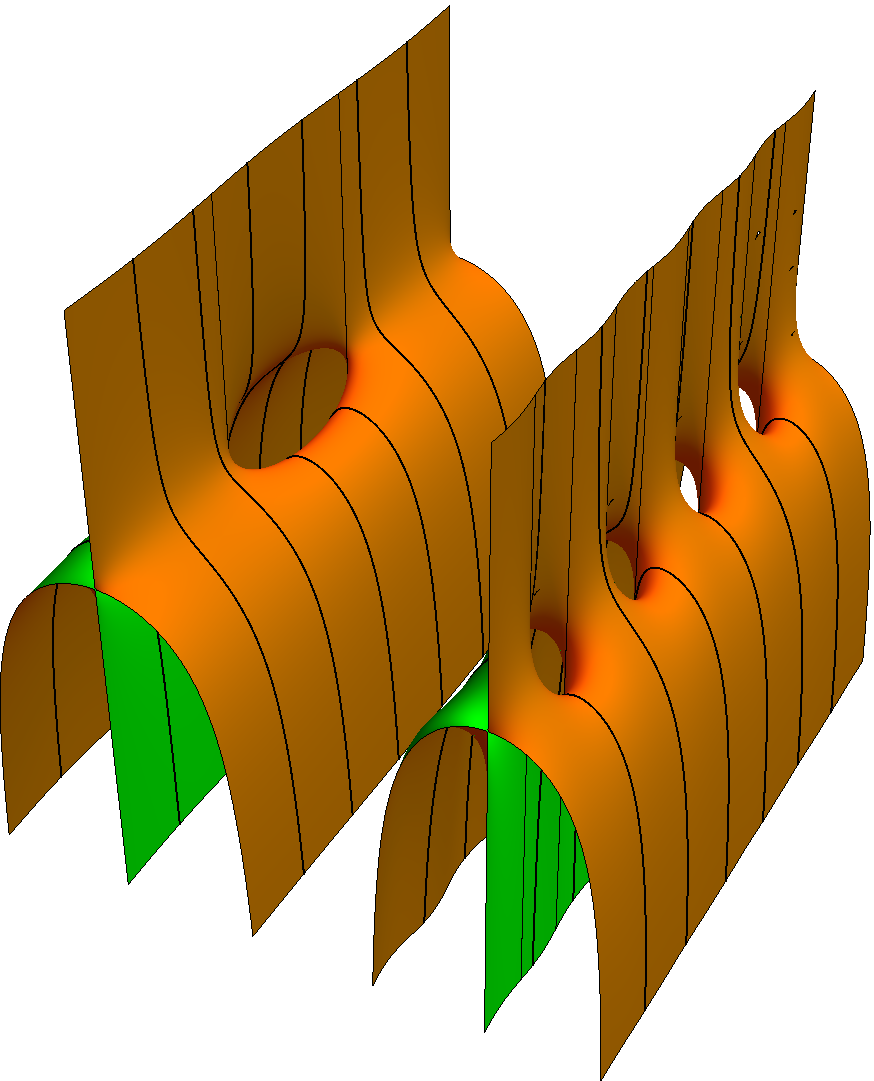}}
\caption{\small \sffamily Two examples of tridents, one with neck-size $a=1$ (right) and the other with neck-size $a=3$ (left).}
\end{figure}

\begin{theorem}[Limits in \cite{HMW2} Theorem 26]\label{thm:LimitsTridents}
The limit of tridents $M_a$ as $a \to \infty$ exists and is the same as the limit of Scherkenoids of width $\pi$, the unique pitchfork $\pP_{\pi}$ of width $\pi$.
\end{theorem}
\noindent
In particular, our Theorem \ref{thm:UniquenessPitchforks} confirms \cite[Conjecture 28]{HMW2}, showing that Theorem \ref{thm:LimitsTridents} holds without needing to pass to subsequences.

Another application of our Theorem \ref{thm:UniquenessPitchforks} concerns the classification of collapsed translators,\footnote{A translator is \textit{collapsed} if it is contained in a vertical slab (strip) of finite width.} performed by Gama-Mart\'in-M\o ller \cite{GMM} for entropy less than 4. 
Their uniqueness results can be stated now in a sharper way thanks to the uniqueness of the pitchforks. 
\begin{theorem}[Classification of symmetric collapsed translators of entropy 3 \cite{GMM}]
The pitchfork $\pP_w$ is the unique translator contained in a slab of width $2 w$, with entropy $3$, and with a vertical line of symmetry. 
\end{theorem}
\noindent 
Pitchforks are important in the classification of semigraphical translators (cf. \cite[Theorem 34]{HMW2}) because they have finite entropy; this is equivalent to having finite topology and quadratic area growth, by \cite[Theorem 9.1]{white21}.
They might therefore arise as blowups at singularities of initially smooth, closed surfaces evolving by mean curvature flow.
Helicoids have cubic area growth, hence infinite entropy, so they cannot arise as such blowups; however, our uniqueness result allows us to obtain a minimal surface as a rescaled limit.
\begin{theorem}
The sequence of (unique) helicoidal translators $ \{ \pH_w \}_{w \in (0,\pi)}$, rescaled by their width $w$, converges in the strong sense to the helicoid minimal surface in $\bR^3$ as $w \to 0^+$.
\end{theorem}
\noindent 
The existence of a subsequential limit is established by a compactness argument in \cite[Theorem 1.1]{white18}; we can now say that the limit exists and is unique.
It is clear that the limit is a minimal surface, because the Gauss maps $\nu_{\pH_w}$ of the helicoids converge, after the rescaling, to a direction in the $(x,y)-$plane; hence, $\vec{H} = 0$ in the limit.
This means that the limiting manifold is a minimal surface, which is in fact the helicoid.
Note that \cite[Theorem 2.2]{HMW1} obtains this limit without requiring the uniqueness of helicoidal translators.


\begin{bibdiv}
\begin{biblist}

\bib{chini}{article}{ 
author={Chini, F.},
doi = {10.1515/geofl-2020-0101},
url = {https://doi.org/10.1515/geofl-2020-0101},
title = {Simply connected translating solitons contained in slabs},
journal = {Geometric Flows},
number = {1},
volume = {5},
year = {2020},
pages = {102--120},}

\bib{GMM}{article}{
author={Gama, E.S.},
author={Mart\'{\i}n, F.},
author={M\o ller, N.M.},
   title={Finite entropy translating solitons in slabs},
   journal={Preprint arXiv:2209.01640},
   date={2022},
}

\bib{GT}{book}{
   author={Gilbarg, David},
   author={Trudinger, Neil S.},
   title={Elliptic partial differential equations of second order},
   series={Classics in Mathematics},
   note={Reprint of the 1998 edition},
   publisher={Springer-Verlag, Berlin},
   date={2001},
   pages={xiv+517},
   isbn={3-540-41160-7},
   review={\MR{1814364}},
}

\bib{graphs}{article}{
   author={Hoffman, D.},
   author={Ilmanen, T.},
   author={Mart\'{\i}n, F.},
   author={White, B.},
   title={Graphical translators for mean curvature flow},
   journal={Calc. Var. Partial Differential Equations},
   volume={58},
   date={2019},
   number={4},
   pages={Paper No. 117, 29},
   issn={0944-2669},
   review={\MR{3962912}},
   doi={10.1007/s00526-019-1560-x},
   Label={HIMW19a},
}

\bib{graphs-correction}{article}{
   author={Hoffman, D.},
   author={Ilmanen, T.},
   author={Mart\'{\i}n, F.},
   author={White, B.},
   title={Correction to: Graphical translators for mean curvature flow},
   journal={Calc. Var. Partial Differential Equations},
   volume={58},
   date={2019},
   number={4},
   pages={Art. 158, 1},
   issn={0944-2669},
   review={\MR{4029723}},
   review={Zbl 07091751},
   doi={10.1007/s00526-019-1601-5},
   Label={HIMW19b},
}
		

	

\bib{himw-survey}{article}{
author={Hoffman, D.},
author={Ilmanen, T.},
author={Mart\'{\i}n, F.},
author={White, B.},
title={Notes on Translating Solitons of the Mean Curvature Flow},
   conference={
      title={T. Hoffmann et al. (eds.), Minimal Surfaces: Integrable Systems and Visualisation},
   },
   book={
      series={Springer Proceedings in Mathematics \& Statistics},
      volume={349},
      publisher={Springer Nature Switzerland AG},
      },
     date={2021},
   pages={147--168},
   review={\MR{2167267}},
   doi={10.1007/978-3-030-68541-6\;9},
}


\bib{HMW1}{article}{
   author={Hoffman, D.},
   author={Mart\'{\i}n, F.},
   author={White, B.},
   title={Scherk-like translators for mean curvature flow},
   journal={J. Differential Geom.},
   volume={122},
   date={2022},
   number={3},
   pages={421--465},
   issn={0022-040X},
   review={\MR{4544559}},
   doi={10.4310/jdg/1675712995},
   Label={HMW22a},
}

\bib{HMW2}{article}{
   author={Hoffman, D.},
   author={Mart\'{\i}n, F.},
   author={White, B.},
   title={Nguyen's tridents and the classification of semigraphical
   translators for mean curvature flow},
   journal={J. Reine Angew. Math.},
   volume={786},
   date={2022},
   pages={79--105},
   issn={0075-4102},
   review={\MR{4434751}},
   doi={10.1515/crelle-2022-0005},
   Label={HMW22b},
}

\bib{morserado}{article}{
author={Hoffman, D.},
author={Mart\'{\i}n, F.},
author={White, B.},
title={Morse-Rad\'{o} theory for minimal surfaces},
volume={108},
number={4},
pages={1669-1700},
date={2023},
journal={J. Lond. Math. Soc.},
doi={10.1112/jlms.12791},
Label={HMW23a},
}


\bib{HMW3}{article}{
   author={Hoffman, D.},
   author={Mart\'{\i}n, F.},
   author={White, B.},
   title={Translating Annuli for Mean Curvature Flow},
   date={2023},
  journal={Preprint arXiv:2308.02210},
  doi={https://doi.org/10.48550/arXiv.2308.02210},
  Label={HMW23b},
}
                        

\bib{ilmanen_1994}{article}{
   author={Ilmanen, T.},
   title={Elliptic regularization and partial regularity for motion by mean
   curvature},
   journal={Mem. Amer. Math. Soc.},
   volume={108},
   date={1994},
   number={520},
   pages={x+90},
   review={\MR{1196160 (95d:49060)}},
   review={Zbl 0798.35066},
}

\bib{magnanini}{article}{
   author={Magnanini, R.},
   title={An introduction to the study of critical points of solutions of
   elliptic and parabolic equations},
   journal={Rend. Istit. Mat. Univ. Trieste},
   volume={48},
   date={2016},
   pages={121--166},
   issn={0049-4704},
   review={\MR{3592440}},
   doi={10.13137/2464-8728/13154},
}

\bib{MSHS15}{article}{
   author={Mart\'{\i}n, F.},
   author={Savas-Halilaj, A.},
   author={Smoczyk, K.},
   title={On the topology of translating solitons of the mean curvature
   flow},
   journal={Calc. Var. Partial Differential Equations},
   volume={54},
   date={2015},
   number={3},
   pages={2853--2882},
   issn={0944-2669},
   review={\MR{3412395}},
   doi={10.1007/s00526-015-0886-2},
}

\bib{nguyen}{article}{
   author={Nguyen, X. N.},
   title={Translating tridents},
   journal={Comm. Partial Differential Equations},
   volume={34},
   date={2009},
   number={1-3},
   pages={257--280},
   issn={0944-2669},
   review={\MR{2512861}},
   doi={10.1080/03605300902768685},
}



\bib{spruck-xiao}{article}{
   author={Spruck, J.},
   author={Xiao, L.},
   title={Complete translating solitons to the mean curvature flow in $\bR^3$ with nonnegative mean curvature},
   journal={Amer. J. Math.},
   volume={142},
   date={2020},
   number={3},
   pages={993--1015},
   issn={0002-9327},
   review={\MR{4101337}},
   doi={10.1353/ajm.2020.0023},
}



\bib{white18}{article}{
   author={White, B.},
   title={On the compactness theorem for embedded minimal surfaces in
   3-manifolds with locally bounded area and genus},
   journal={Comm. Anal. Geom.},
   volume={26},
   date={2018},
   number={3},
   pages={659--678},
   issn={1019-8385},
   review={\MR{3844118}},
   doi={10.4310/CAG.2018.v26.n3.a7},
}

\bib{white21}{article}{
   author={White, B.},
   title={Mean curvature flow with boundary},
   journal={Ars Inven. Anal.},
   date={2021},
   pages={Paper No. 4, 43},
   review={\MR{4462472}},
}

\end{biblist}

\end{bibdiv}
\end{document}